\documentclass[letter]{amsart}
\usepackage{amsmath,amsfonts,amsthm,amssymb}
\usepackage{hyperref}
\usepackage{graphicx}
\usepackage{color}

\numberwithin{equation}{section}

\theoremstyle{plain}
\newtheorem{theorem}{Theorem}[section]
\newtheorem{hypothesis}{Hypothesis}

\newtheorem{lemma}[theorem]{Lemma}
\newtheorem{proposition}[theorem]{Proposition}
\newtheorem{corollary}[theorem]{Corollary}

\theoremstyle{remark}

\theoremstyle{definition}

\DeclareMathOperator{\tr}{tr}

\DeclareMathOperator{\dd}{d\!}

\newcommand{\R}{\mathbb{R}}
\newcommand{\SF}{\mathbb{S}}
\newcommand{\n}{\textbf{\em n}}

\renewcommand{\labelenumi}{(\roman{enumi})}

\begin{document}

\title[Entire solutions to equivariant elliptic systems]{Entire solutions to equivariant elliptic systems with variational structure}
\author{Nicholas D.\ Alikakos \and Giorgio Fusco}
\address{Department of Mathematics\\ University of Athens\\ Panepistemiopolis\\ 15784 Athens\\ Greece \and Institute for Applied and Computational Mathematics\\ Foundation of Research and Technology -- Hellas\\ 71110 Heraklion\\ Crete\\ Greece} 
\email{\href{mailto:nalikako@math.uoa.gr}{\texttt{nalikako@math.uoa.gr}}}
\address{Dipartimento di Matematica Pura ed Applicata\\ Universit\`a degli Studi dell'Aquila\\ Via Vetoio\\ 67010 Coppito\\ L'Aquila\\ Italy}
\email{\href{mailto:fusco@univaq.it}{\texttt{fusco@univaq.it}}} 
\thanks{The first author was supported by Kapodistrias grant No.\ 15/4/5622 at the University of Athens.}

\begin{abstract}
In the present paper we consider the system $\Delta u - W_u(u) = 0$, where $u: \R^n \to \R^n$, for a class of potentials $W: \R^n \to \R$ that possess several global minima and are invariant under a general finite reflection group $G$. We establish existence of nontrivial $G$-equivariant entire solutions connecting the global minima of $W$ along certain directions at infinity.
\end{abstract}

\maketitle

\section{Introduction}
We consider the system
\begin{equation}\label{system}
\Delta u - W_u(u) = 0, \text{ for } u: \R^n \to \R^n,
\end{equation}
where $W: \R^n \to \R$ and $W_u := (\partial W / \partial u_1, \dots, \partial W / \partial u_n)^{\top}$ is the gradient of $W$. We assume that $W$ has $N\geq 2$ distinct global minima $a_i$, for $i=1,\dots,N$, and address the problem of finding an entire solution $u: \R^n \to \R^n$ of \eqref{system} that {\em connects} the $N$ minima of $W$. That is, a solution of \eqref{system} such that 
\begin{equation}\label{lim-system}
\lim_{\lambda\to+\infty} u(\lambda\eta_i) = a_i, \text{ for } i=1,\dots,N,
\end{equation}
for certain unit vectors $\eta_i\in\SF^{n-1}$, where $\SF^{n-1}\subset\R^n$ is the unit sphere.

System \eqref{system} is formally the Euler--Lagrange equation corresponding to the action
\begin{equation}\label{action}
J(u) = \int_{\R^n} \left\{ \frac{1}{2} |\nabla u|^2 + W(u) \right\} \dd x.
\end{equation}
One of the obstructions in the study of \eqref{system} is that for dimensions $n \geq 2$ the action is infinite for the class of solutions we are interested in (see \cite{a}).

We now list our assumptions on the potential $W$.

\begin{hypothesis}[$N$ nondegenerate global minima]\label{h1}
The potential $W$ is of class $C^2$ and satisfies $W(a_i)=0$, for $i=1,\ldots,N$, and $W>0$ on $\R^n \setminus \{a_1,\dots a_N\}$. Furthermore, there holds $v^\top\partial^2 W(u)v \geq c^2 |v|^2$, for $v\in\R^n$ and $|u-a_i| \leq \bar{q}$, for some $c$, $\bar{q} > 0$, and for $i=1,\ldots, N$.
\end{hypothesis}
We recall some examples of potentials that have been studied in the past. The case $n=1$, $N=2$ is textbook material and the corresponding solution is known as the {\em heteroclinic connection}.  In \cite{bgs}, Bronsard, Gui, and Schatzman constructed a solution for $n=2$, $N=3$, while recently in \cite{gs}, Gui and Schatzman constructed a solution for $n=3$, $N=4$; these last two solutions are known as the {\em triple-junction solution} on the plane and the {\em quadruple-junction solution} in space respectively. Triple-junction and quadruple-junction solutions have additional significance of their own and we will comment on them later.

In all these works (for $n \geq 2$), the potentials $W$ have been assumed to have certain symmetries. This takes us to the next hypothesis.

\begin{hypothesis}[Symmetry]\label{h2}
The potential $W$ is invariant under a finite reflection group $G$ acting on $\R^n$ (Coxeter group), that is,
\begin{equation}\label{g-invariance}
W(gu) = W(u), \text{ for all } g \in G \text{ and } u \in \R^n.
\end{equation}
Moreover, we assume that there exists $M>0$ such that
$W(su) \geq W(u)$, for $s\geq 1$ and $|u|=M.$
\end{hypothesis}
We seek {\em equivariant} solutions of system \eqref{system}, that is, solutions satisfying
\begin{equation}\label{equivariance}
u(gx) = gu(x), \text{ for all } g \in G \text{ and } x \in \R^n.
\end{equation}
In \cite{bgs} $G=\mathcal{H}^{3}_{2}$, the group of symmetries of the equilateral triangle, with six elements, and in \cite{gs} $G=\mathcal{T}^{*}$, the group of symmetries of the tetrahedron, with twenty four elements.

The hypothesis next relates the number and location of the minima of $W$ to the group $G$. If $\mathcal{G}$ is a group, we denote by $|\mathcal{G}|$ the order of  $\mathcal{G}$.

\begin{hypothesis}[Location and number of global minima]\label{h3}
Let $F \subset \R^n$ be a fundamental region\footnote{See \cite{gb} or \cite{hu} and Section \ref{alg-pre}.} of $G$. We assume that $\overline{F}$ (the closure of $F$) contains a single global minimum of $W,$ say $a_1$, and let $G_{a_1}$ be the subgroup of $G$ that leaves $a_1$ fixed. Then, as it follows by the invariance of $W$, the number of the minima
of $W$ is
\begin{equation}\label{enne}
N= \frac{|G|}{|G_{a_1}|}.
\end{equation}
\end{hypothesis}
We give here some examples. For $\mathcal{H}^{3}_{2}$ on the plane, we can take as $F$ the $\frac{\pi}{3}$ sector. If $a_1 \in F$, then $N=6$, while if $a_1$ is on the walls, then $N=3$. In higher dimensions we have more options since we can place $a_1$ in the interior of $\overline{F}$, in the interior of a face, on an edge, and so on. For example, if $G=\mathcal{W}^*$, the group of symmetries of the cube in three-dimensional space, then $|G|=48$. If the cube is situated with its center at the origin and its vertices at the eight points $(\pm 1, \pm 1, \pm 1)$, then we can take as $F$ the simplex generated by $s_1 = e_1 + e_2 + e_3$, $s_2 = e_2 + e_3$, and $s_3 = e_3$, where the $e_i$'s are the standard basis vectors. We have then the following options:
\begin{enumerate}
\item On the edge $s_3$, $N=6$.
\item On the edge $s_1$, $N=8$.
\item On the edge $s_2$, $N=12$.
\item In the interior of a face, $N=24$.
\item In the interior of the fundamental region, $N=48$.
\end{enumerate}

The hypotheses so far have been purely geometric. Our final hypothesis is analytic.

\begin{hypothesis}[$Q$-monotonicity]\label{h4}
We restrict ourselves to potentials $W$ for which there is a continuous function $Q: \R^n \to \R$,
which, for some constants $C_{\pm} > 0$ and a $C^2$ function $H: \R^n \to \R$, such that $H(0) = 0$ and $H_u(0) = 0$, satisfies
\begin{subequations}\label{q-list}
\begin{align}
&Q \text{ is convex,}\label{q-list-a}\\
&Q(g u)=Q(u), \text{ for } u\in D,\ g\in G_{a_1},\label{q-list-b}\\
&Q(u + a_1) = |u| + H(u),\label{q-list-c}\\
&Q(u) >0 \text{ and } C_{-} \leq |Q_u(u)| \leq C_{+}, \text{ on } \R^n \setminus \{ a_1 \}\label{q-list-d},
\end{align}
\end{subequations}
and, moreover,
\begin{equation}\label{q-monotonicity}
\left\langle Q_u(u), W_u(u) \right\rangle \geq 0, \text{ in } D\setminus \{ a_1 \},
\end{equation}
where we have set
\begin{equation}
D:= \mathrm{Int}\left( {\cup_{g\in G_{a_1}} g\overline{F}} \right).
\end{equation}
\end{hypothesis}

For $n=1$ and even symmetry, for a double-well potential $W$, and $D=F=\{ u>0 \}$, $Q$-monotonicity implies that $W_u(u)(u-a_1)\geq 0$, for $u>0$.  

For $G=\mathcal{H}^{3}_{2}$ on the plane, $F$ the $\frac{\pi}{3}$ sector, and $a_1=(1,0)$, it can be verified that the triple-well potential
\[ W(u_1,u_2) = |u|^4 + 2 u_1 u_{2}^{2} - \frac{2}{3} u_{1}^{3} - |u|^2 + \frac{2}{3} \]
satisfies the $Q$-monotonicity condition in $D=\{ (r,\theta) \mid r>0,\ \theta \in (-\frac{\pi}{3},\frac{\pi}{3})\}$, with $Q(u)=|u - a_1|$, where $u=(u_1,u_2)$.

For $n=3$, $G=\mathcal{T}^*$, $F$ the simplicial cone generated by $(\sqrt{2/3},\, 0,\, 1/\sqrt{3})$, $(0,\, \sqrt{2/3},\, 1/\sqrt{3})$, $(0,\, 0,\, 1/\sqrt{3})$, and $a_1=(\sqrt{2/3},\, 0,\, 1/\sqrt{3})$, we can take as an example the quadruple-well potential
\[ W(u_1,u_2,u_3) = |u|^4 - \frac{4}{\sqrt{3}} (u^{2}_{1} - u^{2}_{2})u_3 - \frac{2}{3} |u|^2 + \frac{5}{9},\]
with $Q(u) = |u-a_1|$, where $u=(u_1, u_2, u_3)$, and $D$ the simplicial cone generated by
$(0,\, \sqrt{2/3},\, 1/\sqrt{3})$, $(0,\, -\sqrt{2/3},\, 1/\sqrt{3})$, $(\sqrt{2/3},\, 0,\, -1/\sqrt{3})$.

As a final example, take $G$ to be the reflection group on $\R^n$ generated by the coordinate planes, $F$ the simplicial cone generated by the standard basis $e_1 = (1,\ldots,0)$, \ldots, $e_n = (0,\ldots,1)$, and $a_1 = (\alpha_1, \ldots, \alpha_n)$, for $\alpha_i >0$. Then, the potential
\[
W(u) = \sum_{k=1}^{n} C_k\left( u_{k}^{2} ( u_{k}^{2} - 2 \alpha_{k}^{2} ) + \alpha_{k}^{4} \right), \text{ for } u=(u_1, \ldots,u_n) \in \R^n,
\]
where $C_k$ are given positive constants, satisfies the $Q$-monotonicity condition in $D=F$ with $Q=|u-a_1|$. Note that in this last example $a_1$ is in the interior of $\overline{F}$ and, therefore, $N=|G|=2^n.$

We refer to \cite[Proposition 1]{af2} for the details of the construction of the triple-well potential above, as well as for information on the construction of potentials in general. In \cite[Proposition 3]{af2} it is established that for any given reflection group $G$ there exist infinitely-many smooth potentials $W$ satisfying Hypotheses \ref{h1}--\ref{h4}.

Next we explain\footnote{Since $Q$ is not smooth at $a_1$ by \eqref{q-list-c}, the calculations below should be interpreted in the distributional sense: for $u\in L_{\mathrm{loc}}^1(\R^n,\R^n)$, $\Delta u\in L_{\mathrm{loc}}^1(\R^n,\R^n)$, we have  
\[ \Delta(Q(u(x))) \geq \langle \Delta u(x), Q_u(u(x))\rangle,\]
with the convention that $Q_u(0)=0$. This is a straightforward extension of the well-known Kato inequality (see \cite[p.\ 85]{hs}). We thank Alberto Farina for suggesting the relationship.} how the $Q$-monotonicity is utilized in the proof. If $u$ is $C^2$, then
\begin{equation}\label{trace}
\Delta Q(u(x)) = \tr \left\{ (\partial^2 Q(u(x))) (\nabla u(x)) (\nabla u(x))^\top \right\} + \left\langle Q_u (u(x)), \Delta u(x) \right\rangle,
\end{equation}
where $(\partial^2 Q)$  stands for the Hessian of $Q$. If now $u$ has the property
\begin{equation}\label{d-positivity}
u(\overline{F})\subset\overline{F}\quad \text{(positivity)},
\end{equation}
then $u(\overline{D})\subset\overline{D}$, and from \eqref{trace} and convexity it follows that
\begin{equation}\label{positivity}
\Delta Q(u(x)) \geq \langle Q_u (u(x)), \Delta u(x)\rangle,
\end{equation}
and, if $u$ is a solution of \eqref{system}, for $x \in D$ we have 
\begin{equation}\label{positivity1} 
\Delta Q(u(x)) \geq \langle Q_u (u(x)), W_u (u(x)) \rangle \geq 0, 
\end{equation}
from \eqref{q-monotonicity}. Subharmonicity then provides in $D$ a first global estimate on $|u-a_1|$. Hence, a key step is to show that the candidate solution $u$ is a {\em positive} map, that is, that it satisfies \eqref{d-positivity}.

We now proceed with the statement of the main results.
\begin{theorem}\label{theorem1}
Under Hypotheses \ref{h1}--\ref{h4}, there exists an equivariant classical solution to system \eqref{system} such that
\begin{enumerate}
\item $|u(x)-a_1| \leq K \mathrm{e}^{-k d(x,\partial D)}$, for $x \in D$ and for positive constants $k$, $K,$ \medskip
\item $u(F) \subset F$.
\end{enumerate}
In particular, $u$ connects the $N=|G|/|G_{a_1}|$ global minima of $W$:
\[ \lim_{\lambda \to +\infty} u(\lambda g \eta) = g a_1, \text{ for all } g \in G,\]
uniformly for $\eta$ in compact subsets of $D\cap\SF^{n-1}.$
\end{theorem}
We let $B_{x,R}$ be the ball of radius $R>0$ centered at $x \in \R^n$ and $B_R$ be the ball of radius $R>0$ centered at the origin; for $A \subset \R^n$ we set $A_R = A \cap B_R$ and for $A,B \subset \R^n$ we let $A+B = \{ a+b \mid a \in A,\, b \in B \}$. We denote by $W_{\mathrm{E}}^{1,2}(B_R;\R^n)$ the subspace of $W^{1,2}(B_R;\R^n)$ of the maps that satisfy the equivariance condition \eqref{equivariance} for $x \in B_R$. 

The proof of Theorem \ref{theorem1} is based on a family of constrained minimization problems
\begin{equation}\label{min-problems0}
\min_{\mathcal{A}^R} J_{B_R}, \text{ where } J_{B_R}(u) = \int_{B_R} \left\{ \frac{1}{2} |\nabla u|^2 + W(u) \right\} \dd x,
\end{equation}
over the set $\mathcal{A}^R\subset W_\mathrm{E}^{1,2}(B_R,\R^n)$ of {\em admissible} maps which is defined in \eqref{admissible-set}. The admissible set $\mathcal{A}^R\subset W_\mathrm{E}^{1,2}(B_R,\R^n)$ is defined by imposing two constraints: the constraint of positivity \eqref{d-positivity} and the pointwise bound
\begin{equation}\label{bo1}
|u(x)-a_1| \leq q_0 < \bar{q}, \text{ for } x\in\Omega^R + B_{\delta^{\prime}\!/2},
\end{equation}
where $\bar{q}$ is the constant in Hypothesis 1, $\Omega^R \subset D_R$ is defined in \eqref{omegar-def}, and $q_0$, $\delta^{\prime}$ are suitable positive constants.

Problem \eqref{min-problems0} provides a family of minimizers $\{ u_R \in\mathcal{A}^R \}$; we seek then to construct the solution by taking the limit, that is,
\begin{equation}\label{limit-solution}
u(x) = \lim_{R\to\infty} u_R(x).
\end{equation}
For carrying out this procedure and to show that the constraints imposed by membership in $\mathcal{A}^R$ are inactive, we need uniform estimates in $R$. 

Our proof consists of a continuity argument (topological part) and a PDE part. The continuity argument is concerned with positivity; it utilizes the gradient flow
\begin{equation}\label{evolution-problem}
\begin{cases}
\dfrac{\partial u}{\partial t} = \Delta u - W_u(u), \text{ in } B_R \times (0,\infty),\bigskip\\
\dfrac{\partial u}{\partial \n} = 0,  \text{ on } \partial B_R \times (0,\infty), \text{ where } {\partial}/{\partial \n} \text{ is the normal derivative},\bigskip\\
u(x,0) = u_0(x), \text{ in } B_R,
\end{cases}
\end{equation}
in the Sobolev space of equivariant maps $W^{1,2}_{\mathrm{E}}(B_R; \R^n)$. We let $t\to u(\cdot,t,u_0)$ be the solution of \eqref{evolution-problem}. We establish that the set of positive maps (in the class of equivariant Sobolev maps)
\begin{equation}\label{pos}
\mathcal{U}^\mathrm{Pos}:=\big\{ u \in W^{1,2}_{\mathrm{E}}(B_R;\R^n) \mid u ( \overline{F_R} ) \subset \overline{F} \big\}
\end{equation}
is strongly (positively) invariant under the flow \eqref{evolution-problem} meaning that $u_0(\overline{F_R})\cap F\neq\varnothing$ implies $u(\cdot,t,u_0)\in\mathcal{U}_0^{\mathrm Pos}$, for $t>0$, where 
\begin{equation}\label{pos-0}
\mathcal{U}_0^\mathrm{Pos} :=\big\{ u \in W^{1,2}_{\mathrm{E}}(B_R;\R^n) \mid u ( \overline{F_R} \cap F ) \subset  F  \big\}.
\end{equation}
With the help of this strong invariance, we establish that there exists an $R_0 > 0$, such that for $R > R_0$ the minimization problem \eqref{min-problems0} has a solution that satisfies the Euler--Lagrange equation $\Delta u - W_u(u) = 0$ in $B_R$. We do not know if minimizing freely without restricting to the set of positive maps will automatically render a positive map.  

The PDE part of the proof is concerned with the pointwise estimates leading to the exponential estimate in Theorem \ref{theorem1}. To indicate the main ideas we assume $Q(u)=|u-a_1|$ and set $q^{u_R}=Q(u_R).$ By positivity \eqref{d-positivity} and by \eqref{positivity},
\begin{equation}\label{pos1}
\Delta q^{u_R} \geq 0, \text{ in } D_R.
\end{equation}
On the other hand, by the nondegeneracy condition in Hypothesis \ref{h1}, we have
\begin{equation}\label{pos2}
\Delta q^{u_R} \geq c^2 q^{u_R}, \text{ where } q^{u_R}\leq\bar{q}.
\end{equation}
Estimate \eqref{pos1} provides a first global bound on $q^{u_R}$ in $D_R$, while estimate \eqref{pos2} implies a stronger exponential bound on $q^{u_R}$ in $\Omega^R$.   For general $Q$ we have to develop first a global coordinate system in $\R^n$ in terms of the level sets of $Q$. By suitably combining \eqref{pos1} and \eqref{pos2} we can construct a local comparison function that enforces (uniformly in $R$) the estimate $|u(x)-a_1| \leq K \mathrm{e}^{-kd(x,\partial D_R)}$, for $x \in D_R$.
  
Previous works on special cases of major interest are \cite{bgs} and \cite{gs}. Our approach and point of view are different and, in particular, we work with a different set of assumptions. In \cite{bgs} and \cite{gs} the authors proceed via Dirichlet problems and build up a higher-dimensional object out of lower-dimensional solutions. We instead proceed via minimization with two constraints. The solution we construct is a global minimizer of $J_{B_R}$ in the class of positive maps satisfying in addition \eqref{bo1}. The positivity constraint is removed via the gradient flow. The other constraint is removed via comparison arguments. We note that by the results of Palais \cite{palais}, equivariance is not a constraint, in the sense that a critical point in the equivariance class is automatically a critical point in $W^{1,2}(\R^n;\R^n)$. The paper \cite{af1} contains some seeds of the present work. 

Symmetry is a rather restrictive assumption. On the other hand, for general potentials that are only required to satisfy Hypothesis \ref{h1}, it may be impossible to characterize a solution of \eqref{system} and \eqref{lim-system} via minimization of the action. Indeed, some of the solutions given by Theorem \ref{theorem1} are expected to be unstable with respect to compact nonsymmetric perturbations. Particular cases where the existence of solutions of \eqref{system} and \eqref{lim-system} has been established without assuming symmetry are for $N=2$, $n\geq 1$, studied in Sternberg \cite{ste}, and also in \cite{af3}, and for $N=3$, $n=2$ in S\'aez Trumper \cite{st1}, where the existence of a triple junction is shown by utilizing the gradient flow. A possible approach for removing {\em a posteriori} the assumption of symmetry could be to establish the stability of the constructed solution in the class of general compact perturbations. This is reasonable for at least those solutions in Theorem \ref{theorem1} which enjoy extra minimality properties (as, for example, the triple-junction solution).  Finally, in light of \cite{af1}, uniqueness should not be expected in general.

The scalar problem related to \eqref{system}, for $u:\R^n \to \R$, and without any symmetry hypotheses on the solution, has been the object of intensive investigation for many years, with the De Giorgi conjecture and the related contributions at the center of this activity (see the expository article of Farina and Valdinoci \cite{fv}).  On the physical side, we note that for describing coexistence of three or more phases ($N \geq 3$), a vector-order parameter $u$ is needed. A triple-well potential in $\R^2$ or a quadruple-well potential in $\R^3$ would be appropriate  for modeling coexistence of three or four phases correspondingly, with the origin $x=0$ representing the coexistence point (or junction). On the geometric side, the rescaled solution $u_\varepsilon (x) := u( x / \varepsilon)$ in the triple and quadruple-well cases is expected to converge, as $\varepsilon \to 0$, to the solution of the corresponding partitioning problem (see Baldo \cite{b}). The boundaries of the partitioning sets form a system of weighted minimal surfaces meeting in groups of three along free-boundary curves called `liquid edges', and liquid edges meet in groups of four at `supersingular' points, the coexistence points mentioned above (cf.\ Dierkes {\em et al.} \cite[\S 4.10.7]{dhkw1}).

The relevance of the solutions of \eqref{system} in the description of the neighborhood of the junction was first pointed out in Bronsard and Reitich \cite{br}, where also the formal linking of the diffused and sharp-interface models was established for $n=2$. For rigorous linking, for $n=2$, see S\'aez Trumper \cite{st2}. For the associated sharp-interface evolution problem involving motion by mean curvature and Plateau angle conditions see \cite{br}, for $n=2$ in the classical smooth evolutions. See also Mantegazza, Novaga, and Tortorelli \cite{mnt} for initiating and partially resolving globally in time the triple-junction case for $n=2$, and Freire \cite{f1}, Schn\"urer and Schulze \cite{ss}, and Schn\"urer {\em et al.} \cite{setal} for related work for $n=2$. For the evolution problem for general $n$ see Freire \cite{f2}.

The paper is structured as follows. In Section \ref{positivity-property} we establish the strong positivity property of the semigroup that \eqref{evolution-problem} generates. In Section \ref{coordinate-system} we introduce the $Q$-coordinate system and in Sections \ref{comparison-function} and \ref{replacement} we state and prove the comparison lemmas needed for deriving the estimate (i) in Theorem \ref{theorem1}. Finally, in Section \ref{proof} we give the proof of Theorem \ref{theorem1}  .

\section{The positivity property}\label{positivity-property}
\subsection{Algebraic preliminaries}\label{alg-pre}
For the general theory of reflection groups we refer to \cite{gb} and \cite{hu}.
Let $G$ be a {\em Coxeter group}, that is, a finite effective subgroup of the orthogonal group $O(\R^n)$, generated by a set of reflections. A reflection $\gamma\in G$ is associated to the hyperplane $\pi_\gamma = \{ x\in\R^n \mid \langle x,\eta_\gamma \rangle = 0 \}$ via
\begin{equation}
\gamma x = x-2 \langle x, \eta_\gamma \rangle \eta_\gamma, \text{ for } x\in\R^n,
\end{equation}
where $\eta_\gamma\in\SF^{n-1}$ is a unit vector.
Every finite subgroup of $O(\R^n)$ has a {\em fundamental region}, that is, a subset $F\subset\R^n$ with the following properties:
\begin{enumerate}
\item $F$ is open and convex,
\item $F\cap gF = \varnothing$, for $I \neq g\in G$, where $I$ is the identity,
\item $\R^n = \cup\{g\overline{F} \mid g\in G\}$.
\end{enumerate}
We choose the orientation of $\eta_\gamma$ so that $F\subset\mathcal{P}_\gamma^+$, where  $\mathcal{P}_\gamma^+=\{x\in\R^n \mid \langle x,\eta_\gamma \rangle > 0\}$. Then, we have
\begin{equation}\label{fundamental-region}
F=\cap_{\gamma\in\Gamma}\mathcal{P}_\gamma^+,
\end{equation}
where $\Gamma\subset G$ is the set of all reflections in $G$.
Given $A\subset\R^n$, the (pointwise) {\em stabilizer} of $A$, denoted by $\mathrm{Stab}[A]$, is the subgroup of $G$ that fixes $A$ pointwise, that is,
\begin{equation}
\mathrm{Stab}[A]=\{g\in G \mid g x=x, \text{ for all } x\in A\}.
\end{equation}
$\mathrm{Stab}[A]$ is the reflection group generated by the reflections that it contains (\cite[p.\ 23]{hu}). In particular, $G_{a_1}$ defined in Hypothesis \ref{h3} is a reflection group. For $A\subset\R^n$ a nonempty set, we also define $G_A\subset G$ to be the subgroup that leaves $A$ fixed as a set, that is,
\begin{equation}\label{stabilizer}
G_A=\{g\in G \mid g A=A \}.
\end{equation}
We conclude this section with a characterization of $G_D$.
\begin{lemma}\label{lemma-g}
There holds
\begin{equation}\label{ga1-gd}
G_{a_1}=G_D.
\end{equation}
\end{lemma}
\begin{proof}
Observe that $G_D=G_{\overline{D}}$ and that by definition, $\overline{D}=\cup\{g\overline{F} \mid g\in G_{a_1}\}$. It follows that
\begin{equation}
g\overline{D}=\overline{D}, \text{ for all } g\in G_{a_1}, 
\end{equation}
and, therefore, that $G_{a_1}\subset G_{\overline{D}}$. To show that $G_{\overline{D}}\subset G_{a_1}$, we note that, by property (ii) of the fundamental region, there is a one-to-one correspondence between $G_{a_1}$ and the orbit  $\{g\overline{F} \mid g\in G_{a_1}\}$ of $\overline{F}$ under $G_{a_1}$. Therefore, $g^\prime\in G\setminus G_{a_1}$ implies $g^\prime\overline{F} \not\in \{g\overline{F} \mid g\in G_{a_1}\}$ and, in turn, $g^\prime\overline{D}\neq\overline{D}$.
\end{proof}

\subsection{Parabolic flows and positivity}   
We can assume  that  $W$ is a $C^2$ potential satisfying the global bound
\begin{equation}\label{global-bound}
|\partial^{2}_{u_i u_j} W(u)| < C, \text{ in } \R^n.
\end{equation}
This can be imposed without loss of generality because of the {\em a priori} pointwise bound \eqref{pointwise-bound}. As before, we denote by $u(\cdot,t;u_0)$ the solution of \eqref{evolution-problem} and let  $\mathcal{U}^\mathrm{Pos}$ and $\mathcal{U}_0^\mathrm{Pos}$ be the sets of equivariant positive and strongly positive maps defined in \eqref{pos} and \eqref{pos-0}. 
 
\begin{theorem}\label{theorem-2-1}
Suppose $W$ satisfies the bound \eqref{global-bound} and the symmetry \eqref{g-invariance}. Then, \eqref{evolution-problem} leaves the positive class $\mathcal{U}^\mathrm{Pos}$ invariant, that is,
\[ {\mathcal{U}}^\mathrm{Pos} \ni u_0 \mapsto u(\cdot,t;u_0) \in \mathcal{U}^\mathrm{Pos},\]
and, moreover,
\[ u(\cdot,t;u_0) \in \mathcal{U}_0^\mathrm{Pos}, \text{ for } t>0, \]
 provided $u_0(\overline{F_R}) \cap F \neq \varnothing$.
 \end{theorem}

We begin with a lemma.
\begin{lemma}\label{lemma2}
Let $u : B_R \to \R^n$ be an equivariant map. Then, $u$ is a
positive map if and only if
\begin{equation}\label{closure1}
u(\overline{(\mathcal{P}_\gamma^+)_R} ) \subset\overline{\mathcal{P}_\gamma^+}
  , \text{ for all } \gamma\in\Gamma,
\end{equation}
where $(\mathcal{P}_\gamma^+)_R  = \mathcal{P}_\gamma^+\cap B_R$.
\end{lemma}

\begin{proof}
Suppose that \eqref{closure1} holds. Then
\[ u(\overline{F_R}) = u ( \cap_{\gamma\in\Gamma} \overline{(\mathcal{P}_\gamma^+)_R} ) \subset \cap_{\gamma\in\Gamma}\, u (\overline{(\mathcal{P}_\gamma^+)_R}) \subset \cap_{\gamma\in\Gamma}\,\overline{\mathcal{P}_\gamma^+} = \overline{F}.\]
Hence, $u$ is positive.

Conversely, suppose that $u$ is a positive equivariant map on $B_R$. Then, equivalently, $u_{\mathrm{e}}$ defined by
\begin{equation}\label{u-e-def}
u_{\mathrm{e}}(x) := 
\begin{cases}
u(x), &\text{for } x \in B_R\smallskip\\
0,  &\text{for } x \in \R^n \setminus B_R
\end{cases}
\end{equation}
is a positive equivariant map on $\R^n$. For any $g \in G$, we have from equivariance and positivity,
\begin{equation}\label{u-e}
u_{\mathrm{e}} (g(\overline{F})) = g(u_{\mathrm{e}}(\overline{F})) \subset g(\overline{F}).
\end{equation}
Now pick a $\gamma\in\Gamma$ and take an $x \in
 \mathcal{P}_\gamma^+ $ and fix it. There is a $g \in G$, denoted by
$g_x$, such that $x \in g_x(\overline{F})$ and $g_x(F)$ is also a
fundamental region. Since for each fundamenal region $F^\prime$ and for each reflection $\gamma$ we have either $F^\prime\subset\mathcal{P}_\gamma^+$ or $F^\prime\subset-\mathcal{P}_\gamma^+$, we conclude that
\begin{equation}
g_x(\overline{F}) \subset \overline{\mathcal{P}_\gamma^+}.
\end{equation}
Thus, by \eqref{u-e}, $u_{\mathrm{e}} (\overline{\mathcal{P}_\gamma^+}) \subset \overline{\mathcal{P}_\gamma^+}$, and so \eqref{closure1} follows.
\end{proof}

We continue with the
\begin{proof}[Proof of Theorem \ref{theorem-2-1}]
Consider \eqref{evolution-problem} with  $u_0 \in \mathcal{U}^\mathrm{Pos}$. By the regularizing property of the equation, the solution is classical for $t >0$, and by \eqref{global-bound}, it exists globally in time and belongs to $C([0,+\infty); W^{1,2}(B_R; \R^n)) \cap C^1((0, +\infty); C^{2+\alpha}(B_R; \R^n) \cap C(\overline{B_R}; \R^n))$, for some $\alpha \in (0,1)$ (see \cite{h}). Consider a reflection $\gamma\in\Gamma$ and set 
\[
\begin{cases}
\zeta(x,t) = \langle u(x,t,u_0),\eta_\gamma \rangle, &\text{in } B_{R} \times (0,\infty),\\
\zeta_0(x) = \langle u_0(x),\eta_\gamma \rangle, &\text{in } B_{R}.
\end{cases}
\]
By taking the inner product of equation \eqref{evolution-problem} with $\eta_\gamma$, we obtain
\begin{equation}\label{phi-problem}
\begin{cases}
\dfrac{\partial \zeta}{\partial t} = \Delta \zeta + c\zeta, &\text{in } B_{R} \times (0,\infty),\bigskip\\
\dfrac{\partial \zeta}{\partial \n} =0, &\text{on } \partial B_{R} \times (0,\infty),\bigskip\\
\zeta(\cdot,0) = \zeta_0,
\end{cases}
\end{equation}
where we have set 
\[ c(x,t)=\dfrac{\langle W_u(u(x,t,u_0),\eta_\gamma\rangle}{\zeta(x,t)}.\]

From  the equivariance of $u(\cdot,t,u_0)$ and $W_u(\gamma u) = \gamma W_u(u)$ it follows that 
\begin{align}
\zeta(x,t) &= -\zeta(\gamma x,t), \text{ in } B_{R} \times (0,\infty),\label{zeta-invariance} \\
c(x,t) &= c(\gamma x,t), \text{ in } B_{R} \times (0,\infty) \label{c-invariance}.
\end{align} 
From the symmetry of $W$ we also have that $u \in \pi_\gamma$ implies  $W_u(u) \in \pi_\gamma$. From this we deduce
\begin{equation}
\langle W_u(u),\eta_\gamma \rangle = \langle u, \eta_\gamma \rangle \left\langle
\int_0^1W_{uu}\big(u + (s-1)\langle u, \eta_\gamma \rangle
\eta_\gamma\big) \eta_\gamma \dd s,\, \eta_\gamma \right\rangle.
\end{equation}
Thus, the coefficient $c(x,t)$ of $\zeta$ in \eqref{phi-problem} is bounded (actually continuous) on $B_{R}\times(0,\infty)$.

Since $u_0$ is a positive map, we have $\zeta_0 \geq 0$ for $\langle x,\eta_\gamma \rangle\geq 0$.   Therefore, by Lemma \ref{lemma2}, for establishing positivity it is sufficient to show that $\zeta(x,t) \geq 0$, for $x\in B_R^+ = \{ x \in B_R \mid \langle x,\eta_\gamma \rangle> 0 \} $ and $t \geq 0$. We note that by \eqref{zeta-invariance} there holds $\zeta(x,t) = 0$ for $x\in \pi_\gamma \times [0,\infty)$, hence if $\zeta$ is a classical solution of \eqref{phi-problem}, we have that $\zeta(x,t)$ is nonnegative on $B_{R}^{+} \times [0,\infty)$ by the maximum principle. Since mollification preserves positivity \cite{eg} and symmetry, the general case follows by continuous dependence in $W^{1,2} (B_R; \R^n)$ for \eqref{phi-problem} (see \cite{h}).

Finally, since $\zeta(x,t)=0$ for $x \in \pi_\gamma \times (0,\infty)$ and since $\zeta(\cdot,t) \in C^{2+\alpha}(B_R)\cap C(\overline{B_R})$ for $t>0$, the Hopf boundary lemma applies to render that
\[
\zeta(x,t) > 0, \text{ in } B_{R}^{+} \times (0,\infty),
\]
unless $\zeta(x,t) \equiv 0$, hence unless $\zeta_0(x) \equiv 0$. But the hypothesis $u_0(\overline{F_R}) \cap F \neq \varnothing$ excludes this second option.
\end{proof}

\section{The coordinate system}\label{coordinate-system}
\begin{lemma}\label{coord-lemma}
Suppose that $Q: \R^n \to \R$ satisfies \eqref{q-list} in Hypothesis \ref{h4}. Then, the following hold.
\begin{enumerate}
\item For each $\nu \in \mathbb{S}^{n-1}$, the ODE system
\begin{equation}\label{dotu}
\frac{\dd u}{\dd q} = \frac{Q_u(u)}{\langle Q_u(u), Q_u(u) \rangle}, \text{ for } u \in \R^n \setminus \{ a_1 \},
\end{equation}
has a unique solution $\tilde{u}: (0,+\infty) \to \R^n$ such that 
\begin{equation}\label{limnu}
\lim_{q\to 0+} \tilde{u}(q;\nu) = a_1\quad \text{and} \quad \lim_{q\to 0+}
\frac{\tilde{u}(q;\nu) - a_1}{|\tilde{u}(q;\nu) - a_1|} = \nu.
\end{equation}

\item The map $\tilde{u}$ and its partial derivatives $\tilde{u}_q$, $\tilde{u}_\nu$ with respect to $q$, $\nu$, extend continuously to $q=0$ and
\[
\tilde u(0;\nu)=a_1,\quad\quad \tilde u_q(0;\nu)=\nu, \quad\quad \tilde u_\nu(0;\nu)=0. 
\]
Moreover,
\[
C_-^\prime\leq \vert\tilde{u}_q(q;\nu)\vert\leq C_+^\prime,
\]
with $C_-^\prime = C_-C_+^{-2}$, $C_+^\prime = C_+C_-^{-2}$.

\item It results that
\begin{equation}\label{nu-symmetry}
\tilde{u}(q;g\nu) = g \tilde{u}(q;\nu), \text{ for } \nu\in\SF^{n-1},\, g\in G_D=G_{a_1}.
\end{equation}

\item The map defined through the solution
\begin{equation*}
(q,\nu) \mapsto \tilde{u}(q;\nu),
\end{equation*}
is a $C^2$ diffeomorphism of $(0,+\infty)\times\SF^{n-1}$ onto $\R^n\setminus \{ a_1 \}$.
\end{enumerate}
\end{lemma}

\begin{proof}
For the proof we refer to  \cite[Proposition 2]{af2}. Here  we present a proof under the stronger hypothesis
\[ Q(u) = |u-a_1|, \text{ for } |u-a_1| \leq r_0,\]
with $r_0 >0$ and small. 

From \eqref{dotu} we have that
\[ \dfrac{\dd}{\dd q} Q(\tilde u(q)) = 1.\]
This implies that the left extremum of the interval of existence of $\tilde u$ is $q=0$ and, furthermore, that
\begin{equation}
\lim_{q\to 0+} \tilde u(q) = a_1.
\end{equation}
Moreover, for $|u-a_1| \leq r_0$  we have that $Q_u(u)={(u-a_1)}/{\vert u-a_1\vert}$ and \eqref{dotu} takes the form ${\dd u}/{\dd q} = (u-a_1)/{\vert u-a_1\vert}$. Therefore, 
\[ \dfrac{\dd}{\dd q} \frac{\tilde u-a_1}{|\tilde u-a_1|} = 0 ,\]
hence, the existence of the second limit in \eqref{limnu} follows. Statements (ii) and (iv) follow by standard ODE theory. Uniqueness and \eqref{q-list-b} imply (iii).
\end{proof}

We regard the pair $(q,\nu)$ as the {\em polar} coordinates of $u=\tilde{u}(q;\nu)$ and associate to the potential $W$ the function $V: (0,+\infty) \times \SF^{n-1} \to \R$ defined by
\begin{equation}\label{v-def}
V(q,\nu) := W (\tilde{u}(q;\nu)),\text{ for } (q,\nu) \in (0,+\infty) \times \SF^{n-1}.
\end{equation}
From \eqref{nu-symmetry} and \eqref{g-invariance} it follows
\begin{equation}
V(q,g\nu) = V(q,\nu),\text{ for } (q,\nu) \in (0,+\infty) \times \SF^{n-1},\, g \in G_{D}.
\end{equation}
We denote by $\Sigma \subset (0,+\infty) \times \SF^{n-1}$ the inverse image of $D\setminus\{a_1\}$ via the diffeomorphism $(q,\nu) \to \tilde{u}(q;\nu)$. The set $\Sigma$ is of the form
\begin{equation}\label{sigma-structure}
\Sigma = \{ (q,\nu) \mid q \in (0,q_\nu),\, \nu \in \mathbb{S}^{n-1} \},
\end{equation}
where, for each $\nu \in \SF^{n-1}$,  $(0,q_\nu)$ is the interval the map $q \to \tilde{u}(q;\nu)$ spends in $D$. We remark that \eqref{q-monotonicity} in Hypothesis \ref{h4} implies, via \eqref{v-def} and \eqref{dotu},
\begin{equation}\label{v-pos}
\frac{\partial V}{\partial q} (q,\nu) \geq 0, \text{ for } (q,\nu)\in\Sigma.
\end{equation}
On the other hand, by Hypothesis \ref{h1},
\begin{equation}\label{vc-pos}
\frac{\partial V}{\partial q} (q,\nu) \geq c^2 \langle \tilde{u}_q(q;\nu), \tilde{u}_q(q;\nu) \rangle p, \text{ for } 0 \leq p \leq q \leq \bar{q},\;\nu\in\SF^{n-1}.
\end{equation}
We show in \eqref{pointwise-bound} and \eqref{bobo} that we can restrict to bounded values of $q$. Therefore, by changing the definition of $V(q,\nu)$ if necessary, we can also assume
\begin{equation}\label{v-pos-1}
\frac{\partial V}{\partial q} (q,\nu) \geq 0, \text{ for } q \gg 1. 
\end{equation}
Given $u \in W^{1,2}(B_R;\R^n)$, set $\mathcal{S}_u := \{x\in B_R \mid u(x) = a_1\}$.
The diffeomorphism defined in Lemma \ref{coord-lemma} associates to the restriction to $\overline{D_R}\setminus {\mathcal{S}_u}$ of any positive equivariant map $u\in\mathcal{U}^{\rm Pos}$ a {\em polar} representation $(q^u,\nu^u) : \overline{D_R}\setminus{\mathcal{S}_u} \to \R\times\SF^{n-1}$ as follows  
\begin{equation}\label{polar-coord}
u\vert_{\overline{D_R}} \leftrightarrow (q^u,\nu^u), \text{ where } u(x)=\tilde u(q^u(x);\nu^u(x)),\, x\in\overline{D_R}\setminus{\mathcal{S}_u}.
\end{equation}
From \eqref{nu-symmetry} and the equivariance of $u$ it follows that the maps $q^u : \overline{D_R}\setminus{\mathcal{S}_u} \to \R^n$ and $\nu^u : \overline{D_R}\setminus{\mathcal{S}_u} \to \SF^{n-1}$ satisfy
\begin{equation}\label{nueq-symmetry}
q^u(g x)=q^u(x) \quad \text{and} \quad \nu^u(g x)=g \nu^u(x),
\end{equation}
for all $x\in\overline{D_R}\setminus{\mathcal{S}_u}$ and all $g\in G_D$.

From (\ref{polar-coord}) we calculate
\[ u_{x_i}(x) = \tilde{u}_q q_{x_i}^u(x) + \tilde{u}_\nu \nu_{x_i}^u(x),\]
thus, utilizing \eqref{diff-nu-1} below,
\begin{equation}\label{nabla-v}
|\nabla u |^2 = \langle \tilde{u}_q, \tilde{u}_q \rangle |\nabla q^u|^2 + \sum_{j=1}^{n} \langle \tilde{u}_\nu \nu_{x_j}^u, \tilde{u}_\nu \nu_{x_j}^u \rangle,
\end{equation}
where $\vert T\vert$ denotes the Euclidean norm of the matrix $T$. From $u\in W^{1,2}(B_R;\R^n)$  it follows that the Euclidean norm $\vert u-a_1\vert$ belongs to $W^{1,2}(B_R;\R)$ hence 
\[ q^u \in W^{1,2} (D_R;\R). \]
From \eqref{nabla-v} and \eqref{v-def} we obtain that, under the standing assumption $u\in\mathcal{U}^{\rm Pos}$, the action takes the form
\begin{align*}
J_{B_R}(u) &=N \int_{D_R} \Big\{ \frac{1}{2} |\nabla u|^2 + W(u) \Big\} \dd x\\
&=N \int_{D_R\cap\{\vert u-a_1\vert>0\}} \Big\{ \frac{1}{2} |\nabla u|^2 + W(u) \Big\} \dd x\nonumber\\
&=N \int_{D_R\cap\{q^u>0\}} \Big\{ \frac{1}{2}\big( \langle \tilde{u}_q, \tilde{u}_q \rangle |\nabla q^u|^2 + \sum_{j=1}^{n} \langle \tilde{u}_\nu \nu^u_{x_j}, \tilde{u}_\nu \nu^u_{x_j} \rangle \big) + V(q^u,\nu^u) \Big\} \dd x,
\end{align*}
where $N= {|G|}/{|G_{a_1}|}$ and we have used $\vert \nabla u\vert=0$ a.e.\ on the measurable set $\{x \mid u(x)=a_1\}$.

\begin{lemma}
Consider the mapping $(q,\nu) \mapsto \tilde{u}(q;\nu)$ as defined in Lemma \ref{coord-lemma}. Then, for any fixed vector $t \perp \nu$, the quadratic form
\begin{equation}\label{quadratic}
\omega (\alpha, \beta) = - \langle \tilde{u}_{qq}, \tilde{u}_{q} \rangle \alpha^2 + \langle \tilde{u}_{q\nu}t, \tilde{u}_{\nu}t \rangle \beta^2 - 2\langle \tilde{u}_{q\nu}t, \tilde{u}_{q} \rangle\alpha\beta, \text{ for } \alpha, \beta \in \R
\end{equation}
is positive semidefinite.
\end{lemma}

\begin{proof}
By differentiating the identity
\begin{equation}\label{idqu}
Q(\tilde{u}(q;\nu)) = q,
\end{equation}
with respect to $q$, we obtain
\begin{equation}\label{idqu-one}
\langle Q_u, \tilde{u}_q \rangle = 1.
\end{equation}
On the other hand, differentiating \eqref{idqu} with respect to $\nu$ in direction $t$, we obtain, using also \eqref{dotu},
\begin{equation}\label{diff-nu-1}
\langle Q_u, \tilde{u}_{\nu}t \rangle = 0 \Leftrightarrow \langle \tilde{u}_q, \tilde{u}_\nu t \rangle = 0,
\end{equation}
and differentiating once more gives
\begin{equation}\label{diff-nu-2}
\langle \tilde{u}_{q\nu}t, \tilde{u}_\nu t \rangle + \langle \tilde{u}_{q}, \tilde{u}_{\nu\nu}(t,t)\rangle = 0.
\end{equation}
Now, differentiating \eqref{idqu-one} with respect to $q$ yields, via \eqref{dotu},
\begin{subequations}\label{diffs}
\begin{equation}\label{diffs-1}
\langle Q_{uu} \tilde{u}_{q}, \tilde{u}_{q} \rangle + \langle Q_u, \tilde{u}_{qq} \rangle = 0 \Leftrightarrow \frac{\langle \tilde{u}_{qq}, \tilde{u}_q \rangle}{\langle \tilde{u}_q, \tilde{u}_{q}\rangle} = - \langle Q_{uu} \tilde{u}_q, \tilde{u}_q \rangle,
\end{equation}
while differentiating with respect to $\nu$ in direction $t$ yields
\begin{equation}\label{diffs-2}
\langle Q_{uu} \tilde{u}_{\nu}t, \tilde{u}_{q} \rangle + \langle Q_u, \tilde{u}_{q\nu}t\rangle = 0 \Leftrightarrow \frac{\langle \tilde{u}_{q\nu}t, \tilde{u}_q \rangle}{\langle \tilde{u}_q, \tilde{u}_{q}\rangle} = - \langle Q_{uu} \tilde{u}_{\nu}t, \tilde{u}_{q} \rangle.
\end{equation}
Finally, differentiating \eqref{diff-nu-1} with respect to $\nu$ yields, using also \eqref{diff-nu-2},
\begin{equation}\label{diffs-4}
\begin{split}
\langle Q_{uu} \tilde{u}_{\nu}t, &\tilde{u}_{\nu}t \rangle +
\langle Q_u, \tilde{u}_{\nu\nu}(t,t)\rangle = 0 \Leftrightarrow\\
&\frac{\langle \tilde{u}_{q\nu}t, \tilde{u}_\nu t \rangle}{\langle
\tilde{u}_q, \tilde{u}_{q}\rangle}=- \frac{\langle
\tilde{u}_{\nu\nu}(t,t), \tilde{u}_q \rangle}{\langle
\tilde{u}_q, \tilde{u}_{q}\rangle} =  \langle
Q_{uu} \tilde{u}_{\nu}t, \tilde{u}_{\nu}t \rangle.
\end{split}
\end{equation}
\end{subequations}
The convexity of $Q$ implies
\begin{equation}\label{q-conv}
\langle Q_{uu} v, v \rangle \geq 0, \text{ for all } v \in \R^n.
\end{equation}
From this and \eqref{diffs-4}, we obtain
\begin{equation}\label{inner1}
\langle \tilde{u}_{q\nu}t, \tilde{u}_{\nu}t \rangle \geq 0,
\end{equation}
while from \eqref{q-conv} and \eqref{diffs-1} we obtain
\begin{equation}\label{inner2}
- \langle \tilde{u}_{qq}, \tilde{u}_{q}\rangle \geq 0.
\end{equation}
From \eqref{q-conv}, by the same argument that proves the Schwarz inequality, we have
\begin{equation}\label{schwarz}
\langle Q_{uu}v, w \rangle^2 \leq \langle Q_{uu}v, v \rangle \langle Q_{uu}w, w \rangle, \text{ for all } v,w \in \R^n.
\end{equation}
Thus, from \eqref{diffs} and \eqref{schwarz}, it follows,
\[ - \langle \tilde{u}_{qq}, \tilde{u}_{q} \rangle \langle \tilde{u}_{q\nu}t, \tilde{u}_{\nu}t \rangle - \langle \tilde{u}_{q\nu}t, \tilde{u}_{q} \rangle^2 \geq 0,\]
which, together with \eqref{inner1} and \eqref{inner2}, concludes the proof.
\end{proof}

\begin{lemma}\label{thre-two}
Assume that $b>0$ and that $u\in\mathcal{U}^{\rm Pos}$  satisfy the following.
\begin{enumerate}
\item The set $A_b \subset D_R$  defined by $A_b: =\{x\in D_R \mid q^u >b\}$ is open,
\item $q^u \in L^\infty(A_b)$ and  $\nu^u : \overline{A_b} \to \mathbb{S}^{n-1}$ is $C^1$ smooth.
\end{enumerate} 
Moreover, let $F: \overline{A_b} \times \R \times \R^n \to \R$ be the function defined by 
\begin{equation}\label{calk1}
F(x,q,z) := 
\frac{1}{2}\bigg\{ \langle \tilde{u}_q(q; \nu^u), \tilde{u}_q(q, \nu^u) \rangle |z|^2 + \sum_{j=1}^{n} \langle \tilde{u}_\nu(q; \nu^u) \nu_{x_j}^u, \tilde{u}_\nu(q, \nu^u) \nu_{x_j}^u \rangle \bigg\},
\end{equation}
for $x \in \overline{A_b}$, $z\in\R^n$, and $q\geq 0$, while for $q < 0$ let
\[ F(x,q,z) := F(x,-q,z). \]

Then, the functionals $\mathcal{K}_{A_b}$ and $\mathcal{E}_{A_b} := \mathcal{K}_{A_b} + \mathcal{V}_{A_b}$, where
\begin{align}\label{calk}
\mathcal{K}_{A_b}(\rho) &:=  \int_{A_b} F(x,\rho,\nabla\rho) \dd x,\\
\mathcal{V}_{A_b} (\rho) &:= \int_{A_b} V(\vert\rho\vert, \nu^u) \dd x,
\end{align}
admit a nonegative minimizer $\rho \in W^{1,2}(A_b) \cap L^\infty(A_b)$ that satisfies the Dirichlet condition $ \rho  = q^u$, for $x\in\partial A_b$ and the invariance condition
\begin{equation}
\rho(g x)=\rho(x), \text{ for } x\in A_b,\, g\in G_{A_b}.
\end{equation}
    
\end{lemma}
\begin{proof}
The smoothness of $\nu^u$ implies that the function $F$ defined in \eqref{calk1} is continuous on $\overline{A_b} \times \R \times \R^n$ and convex in $z$ for each fixed $(x,q) \in \overline{A_b} \times \R$. From this and the boundary condition it follows that $F$ satisfies all assumptions in Theorems 4.5, 4.6 in \cite{giu}. Therefore, the existence of a minimizer $\rho \in W^{1,2}(A_b)$ follows from Theorem 4.6 in \cite{giu}. To show that a minimizer $\rho$ of $\mathcal{K}_{A_b}$ is in $L^\infty(A_b)$ we set $\rho^- :=\min \{\rho,\left\| q^u \right\|_{L^\infty(A_b)}\}$ and observe that
\[
\nabla\rho^- = 0,\text{ on } \{\rho>\rho^-\}
\]
and
\[
\langle \tilde{u}_{q\nu}(q, \nu^u) \nu_{x_j}^u, \tilde{u}_\nu(q, \nu^u) \nu_{x_j}^u \rangle\geq 0\quad (\text{from }\eqref{diffs-4})
\]
imply 
\[
\mathcal{K}_{A_b}(\rho^-)\leq\mathcal{K}_{A_b}(\rho).
\]
The $L^\infty$ bound for a minimizer $\rho$ of $\mathcal{E}_{A_b}$ follows from assumption \eqref{v-pos-1} and a similar argument. Finally, the evenness of $F$ and of $V(\vert\cdot\vert,\nu^u)$ in $q$ imply we can assume $\rho\geq 0$.
\end{proof}

\section{The comparison function $\sigma$}\label{comparison-function}
We prove three lemmas leading to the construction of a map $\sigma$ that we use systematically as a comparison function in the proof of Theorem \ref{theorem1}. We let $\chi_A$ be the characteristic function of a set $A$.
 
Given numbers $l, \lambda>0$, set $L=l+\lambda$ and let
$\varphi=\chi_{\overline{B_l}}\varphi_1+\chi_{\overline{B_L}\setminus\overline{B_l}}\varphi_2$,
where $\varphi_1:\overline{B_l}\to\R$, $\varphi_2:\overline{B_L}\setminus
B_l\to\R$ are defined by
\begin{equation}\label{final-comp}
\begin{cases}
\Delta \varphi_1 = c^2\varphi_1, &\text{in } B_l,\smallskip\\
\varphi_1 = \bar q, &\text{on } \partial B_l,
\end{cases}
\end{equation}
and
\begin{equation}\label{quasifinal-comp}
\begin{cases}
\Delta \varphi_2 = 0, &\text{in } B_L\setminus\overline{B_l},\smallskip\\
\varphi_2 = \bar q, &\text{on } \partial B_l,\smallskip\\
\varphi_2 = \overline{Q}, &\text{on } \partial B_L,
\end{cases}
\end{equation}
where $c$, $\bar q$, and $M$ below are the constants defined in Hypotheses \ref{h1} and \ref{h2} and
\begin{equation}\label{quasifinal-comp5}
\overline{Q} = \max_{u\in\overline{D},\, \vert u\vert\leq M}Q(u),
\end{equation}
(see Hypothesis \ref{h4}). The map $\varphi$ is radial, that is, $\varphi_j(x)=\phi_j(\vert x \vert)$, for $j=1,2.$ Classical properties of Bessel functions imply that  $\phi_1 : [0,l] \to \R$ is positive and increasing together with the first derivative $\phi_1^\prime$. The function $\phi_2:[l,L]\to\R$ is increasing with decreasing first derivative $\phi_2^\prime$, by explicit calculation.

\begin{lemma}\label{lemma-phi1}
The following hold.
\begin{enumerate}
\item The function $\phi_1^\prime(l)$ is strictly increasing for $l\in(0,+\infty)$ and
\begin{equation}\label{phi1-limit}
      \lim_{l\to+\infty}\phi_1^\prime(l) = c\bar q.
    \end{equation}
\item There exists a strictly increasing function $h:(0,+\infty)\to(0,+\infty)$ such that
\begin{equation}\label{ex-bound}
\phi_1(r) \leq \mathrm{e}^{h(l)(r-l)}\phi_1(l),\text{ for } r\in[0,l],
\end{equation}
and $\lim_{l\to+\infty} h(l) = c$.
\item There is a constant $C_0$, independent of $l$, such that
\begin{equation}
\phi_1^{\prime\prime}(r) \leq  C_0,\text{ for } r\in[0,l].
\end{equation}
\end{enumerate}
\end{lemma}

\begin{proof}
(i) and (ii) are proved in \cite[Lemma 2.4]{flp}. From the bound provided by \eqref{ex-bound} for $\phi_1$ and standard arguments it follows that
\begin{equation}
\phi_1^{\prime\prime}(r)\leq C_0,\text{ for } r\in[0,\min\{l,1\}].
\end{equation}
If $l>1,$ from the proof of Lemma 2.4 in \cite{flp}, it follows that $\phi_1^\prime(r)\leq C$, for $r\in[1,l]$, where $C$ is a constant independent of $l$. This together with inequality \eqref{ex-bound} 
imply
\begin{equation}
\phi_1^{\prime\prime}(r)\leq C_0,\text { for } r\in[1,l],\, l>1. \qedhere
\end{equation}
\end{proof}

An explicit computation yields, for $r\in[l,L],$
\begin{equation}\label{phi2-derivative}
\phi_2^\prime(r)=
\begin{cases}
\dfrac{\overline{Q}-\bar q}{r\log(L/l)}, &\text{for } n=2,\medskip\\
(n-2)\dfrac{l^{n-2}(\overline{Q}-\bar q)}{r^{n-1}(1-(l/L)^{n-2})}, &\text{for } n>2.
\end{cases}
\end{equation}

\begin{lemma}\label{lemma-phi2}
The following hold.
\begin{enumerate}
\item Let the ratio $l/L$ be fixed. Then,
\begin{equation}\label{phi2-limit}
\lim_{l\to+\infty}\phi_2^\prime(l) = 0.
\end{equation}
\item Let the difference $L-l=\lambda$ be fixed. Then, $\phi_2^\prime(l)$ is a decreasing function of $l\in(0,+\infty)$ and
\begin{equation}\label{phi2-lim}
\lim_{l\to+\infty}\phi_2^\prime(r) = \frac{\overline{Q}-\bar q}{\lambda},\text{ for } r\in[l,l+\lambda].
\end{equation}
Moreover, there exists a constant $C_0$, independent of $l\in[1,+\infty)$, such that
\begin{equation}
\vert \phi_2^{\prime\prime}(r)\vert \leq \frac{C_0}{l},\text{ for } r\in[l,l+\lambda].
\end{equation}
\end{enumerate}
\end{lemma}

\begin{proof}
(i) is a straightforward consequence of \eqref{phi2-derivative}. We prove (ii) for $n>2.$ The case $n=2$ is similar. To show that $\phi_2^\prime(l)$ is decreasing, we prove that the map $f(l) = l (1- (l / (l+\lambda))^{n-2})$ is increasing. Setting $\xi= l / (l+\lambda)$ we have
\[
f^\prime(l) = d(\xi) := 1-(n-1)\xi^{n-2}+(n-2)\xi^{n-1},\text{ for } \xi\in[0,1),
\]
and $f^\prime(l)>0$, for $l\in(0,+\infty)$, follows from $d(0)=1$, $d(1)=0$, and $d^\prime(\xi)<0$, for $\xi\in(0,1)$. The limit \eqref{phi2-lim} follows from \eqref{phi2-derivative}. The last statement of the lemma follows from
\[ 
\phi_2^{\prime\prime}(r)=-(n-1)\frac{l^{n-1}}{r^n}\phi_2^\prime(l). \qedhere
\]
\end{proof} 

Let $\varphi$ be as before and let $\delta>0$ be a small number. Denote by $\vartheta : B_{l+\delta}\setminus\overline{B_{l-\delta}} \to \R$ the solution of the problem
\begin{equation}\label{small-comp}
\begin{cases}
\Delta \vartheta = 0, &\text{in } B_{l+\delta}\setminus\overline{B_{l-\delta}},\smallskip\\
\vartheta = \varphi, &\text{on } \partial (B_{l+\delta}\setminus
\overline{B_{l-\delta}}).
\end{cases}
\end{equation}
We have $\vartheta(x)=\theta(\vert x\vert))$, where $\theta : [l-\delta, l+\delta] \to \R$ satisfies
\begin{equation}\label{theta-derivative}
\theta^\prime(r)=
\begin{cases}
\dfrac{\phi_2(l+\delta)-\phi_1(l-\delta)}{r\log\frac{l-\delta}{l-\delta}}, &\text{for } n=2,\bigskip\\ 
(n-2)\dfrac{(l-\delta)^{n-2}(\phi_2(l+\delta)-\phi_1(l-\delta))}{r^{n-1}(1-(\frac{l-\delta}{l+\delta})^{n-2})}, &\text{for } n>2.
\end{cases}
\end{equation}

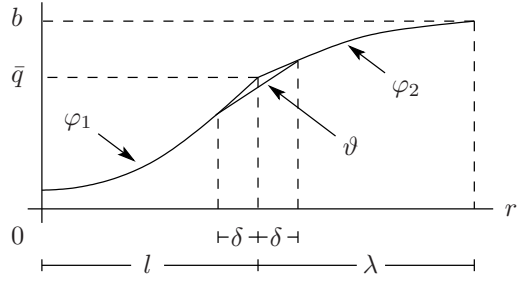
\begin{figure}
\begin{center}
\begin{picture}(0,0)%
\includegraphics{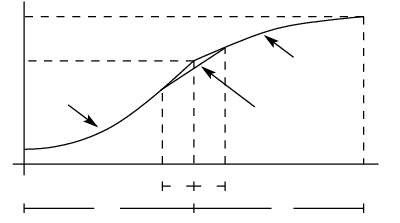}%
\end{picture}%
\setlength{\unitlength}{4144sp}%
%
\begin{picture}(3038,1698)(943,-172)
\put(943,1373){$b$}%
\put(3046,-108){$\lambda$}%
\put(1735,-108){$l$}%
\put(1258,782){$\varphi_1$}%
\put(2926,600){$\vartheta$}%
\put(3221,979){$\varphi_2$}%
\put(3896,234){$r$}%
\put(2490, 52){$\delta$}%
\put(2258, 52){$\delta$}%
\put(943, 65){$0$}%
\put(943,1035){$\bar{q}$}%
\end{picture}%
\caption{The comparison functions $\varphi_1$, $\varphi_2$, and $\vartheta$.}
\label{figure:figure1}
\end{center}
\end{figure}

\begin{lemma}\label{lemma-sigma}
There exist positive constants $l_0,\, \lambda,\, \delta,\, \bar q^\prime < \bar{q},\, \delta^\prime,\, \mu$, such that $l\geq l_0$, $L=l+\lambda$ implies
\begin{enumerate}
\item $\phi_1^\prime(l)>\phi_2^\prime(l)+\mu,$
\item $\vartheta<\varphi, \text{ in } B_{l+\delta}\setminus\overline{B_{l-\delta}}$,
\item The map $\sigma: \overline{B_L}\to\R$ defined by $\sigma=\chi_{B_{l-\delta}\cup (\overline{B_L}\setminus \overline{B_{l+\delta}})}\varphi+\chi_{\overline{B_{l+\delta}}\setminus B_{l-\delta}}\vartheta$ satisfies
\begin{equation}\label{sigma-below}
\sigma\leq \bar q^\prime<\bar q, \text{ in } \overline{B_{l+\delta^\prime}}.
\end{equation}
\end{enumerate}
\end{lemma}

\begin{proof}
Letting the ratio $\rho= l/L$ be fixed, then \eqref{phi1-limit} and \eqref{phi2-limit} imply that there is an $l_0$ such that (i) holds for $l=l_0$ and some $\mu>0$. Fixing $\lambda=l_0 ((\mu/\rho)-1)$, then (i) holds for all $l\geq l_0$. This follows from Lemmas \ref{lemma-phi1} and Lemma \ref{lemma-phi2} (ii), which imply that $\phi_1^\prime(l)$ is increasing and $\phi_2^\prime(l)$ is decreasing for fixed $\lambda$. From \eqref{theta-derivative}, the relation
\[ 
\phi_2(l+\delta)-\phi_1(l-\delta)=(\phi_2^\prime(l)+\phi_1^\prime(l))\delta+o(\delta),
\]
which holds uniformly in $l$ since $\phi_1(l)=\phi_2(l)=\bar{q}$, and 
\begin{align*} 
\log\frac{l+\delta}{l-\delta} &= 2\frac{\delta}{l}+o(\delta),\\
\left( \frac{l-\delta}{l+\delta} \right)^{n-2} &= 1-2(n-2)\frac{\delta}{l}+o(\delta),
\end{align*}
it follows that
\begin{align}\label{theta-phi}
\left\vert \theta^\prime(r)-\frac{1}{2}(\phi_2^\prime(l)+\phi_1^\prime(l)) \right\vert &\leq C\delta, \text{ for } r\in[l-\delta,l+\delta],\\
\vert\theta^{\prime\prime}\vert &\leq \frac{C}{l}, \text{ for } r\in[l-\delta,l+\delta] 	
\end{align}
for some constant $C>0$, independent of $l\in[l_0,+\infty)$. From (i) and \eqref{theta-phi}, and the bounds on $\phi_1^{\prime\prime}$, $\phi_2^{\prime\prime}$, $\theta^{\prime\prime}$, it follows that there is a small $\delta>0$, independent of $l\in[l_0,+\infty)$, such that
\begin{equation*}\label{theta-phi2}
\begin{cases}
\theta^\prime(r) < \phi_1^\prime(r),\text{ for } r\in[l-\delta,l],\smallskip\\
\theta^\prime(r) > \phi_2^\prime(r),\text{ for } r\in[l,l+\delta].
\end{cases}
\end{equation*}
This and $\theta(l-\delta)=\phi_1(l-\delta)$, $\theta(l+\delta)=\phi_2(l+\delta)$, prove (ii). The existence of the number $\bar q^\prime < \bar q$ and $0 < \delta^\prime < \delta$, independent of $l\in[l_0,+\infty)$, follows by the same arguments and from the existence of the limits \eqref{phi1-limit} and \eqref{phi2-lim}.
\end{proof}

\section{The replacement lemmas}\label{replacement}
We divide this section into two parts. In the first part we give conditions on a set $A\subset D_R$ which allow for a map defined on $A$ to be extended to an equivariant map defined on $B_R$. In particular, we analyze the case where $A$ is a ball $B_{x,r}$ and show that, except for a neighborhood of $\partial D_R$, $D_R$ can be covered by balls $B_{x,r}$, with $r \geq L_0 = l_0 + \lambda$, that satisfy the condition ensuring the possibility of equivariant extension. These results are utilized in the second part where we prove Proposition \ref{sigmacomparison} and Proposition \ref{q-comparison} that are basic for showing that $u_R$ satisfies \eqref{bo1} with the sign of strict inequality.

\subsection{Equivariant extension and the set $\Omega^R$}
Let $\Gamma\subset G$ and $\pi_\gamma,\,\gamma \in \Gamma$, and $G_A$ as in Section \ref{alg-pre}. We let $\Gamma_A=\Gamma\cap G_A.$ For $x\in\R^n$, we set $G_x=G_{\{x\}}$, $\Gamma_x=\Gamma_{\{x\}}$. $G_x$ coincides with $\mathrm{Stab}[\{x\}]$ and it is generated by $\Gamma_x$ (see \cite{hu}).

\begin{lemma}\label{algebra}
Let $A$ be an open and connected subset of $\R^n.$ Assume that for all $\gamma \in \Gamma$,
\begin{equation}\label{a-cond}
\gamma A\cap A\neq\varnothing \quad\text{implies}\quad \gamma A=A.
\end{equation}
Then, the following hold. 
\begin{enumerate}
\item For all $g \in G$ 
\begin{equation}\label{ar-cond}
g A\cap A\neq\varnothing \quad\text{implies}\quad g A=A.
\end{equation}
\item $G_A$ is the reflection group generated by  
\begin{equation}
\Gamma_A^* = \{ \gamma\in\Gamma \mid A \cap\pi_\gamma \neq \varnothing \}.
\end{equation}
\end{enumerate}
\end{lemma}
 
\begin{proof} For each pair of fundamental regions $F_a$, $F_b$, there is a unique $g \in G$ that satisfies
\begin{equation}
g F_a=F_b.
\end{equation}
Therefore, if $F_i$, for $1\leq i\leq N$, are the distinct fundamental regions with the property that $A_i = A \cap F_i \neq \varnothing$, there is a unique $g_i\in G$ such that $g_i F_1=F_i$.

\medskip \noindent {\em Step 1}. There exist $\gamma_j\in\Gamma_A^*$, for $1\leq j\leq M$, such that $g_i=\gamma_M\cdots\gamma_1$. Since $A$ is connected, given $x_i \in A_i$, for $1\leq i\leq N$, there is an arc $[0,1] \ni s\to x(s) \in A$, such that $x(0) = x_1$, $x(1)=x_i$. Since $A$ is open, by slightly deforming $x(s)$ if necessary, we can assume that there are sequences $s_j$, for $1\leq j\leq M$, and $A_{i_j}$, for $1\leq j\leq M+1$, such that
\begin{align}
x(s) &\in A_{i_j}, \text{ for } s_{j-1} < s <s_j,\text{ and } 1\leq j\leq M+1,\\
x(s_j) &\in \pi_{\gamma_j}, \text{ for } 1\leq j\leq M, 
\end{align}
where $s_0 = 0$, $s_{M+1} = 1$, and where $\gamma_j$ is the reflection associated to the plane $\pi_{\gamma_j}$ on the common boundary between $F_{i_j}$ and $F_{i_{j+1}}$. This shows that $g_i = \gamma_M \cdots \gamma_1$ and, therefore, that $g_i$ belongs to the group generated by $\Gamma_{A}^{*}$.

\medskip \noindent {\em Step 2}. We now prove that $g=\gamma_M \cdots \gamma_1$, with $\gamma_j\in\Gamma_A^*$, for $1\leq j\leq M$, is a necessary and sufficient condition in order that $g A=A$. From the definition of $\Gamma_A^*$ it is plain that the condition is sufficient. On the other hand, $g A=A$ implies $g F_h=F_k$, for some $1\leq h$, $k\leq N$, and therefore, by Step 1, we have that $g=g_k g_{h}^{-1}$ is the product of reflections in $\Gamma_A^*$.

\medskip \noindent {\em Step 3}. To complete the proof of (i) we show that $g F_i\cap A=\varnothing$ implies $g A\cap A=\varnothing$. Indeed, if this is not the case, there exist $F_h$, $F_k$, such that $g F_h = F_k$. It follows that $g = g_k g_{h}^{-1}$ and therefore, by Step 1, $g F_i = g_k g_{h}^{-1} F_i = F_j$, for some $1\leq j\leq N$, in contradiction with the assumption.
 \end{proof} 

We denote by $\Pi$ the union of all planes $\pi_\gamma$ of all reflections $\gamma\in G$ and define
\begin{equation}\label{plane-invariant-x}
\Pi_x = \Pi \setminus \tilde{\Pi}_x, \text{ where }\tilde{\Pi}_x = \cup_{\gamma\in\Gamma\setminus\Gamma_x} \pi_\gamma.
\end{equation}
Note that $\tilde\Pi_x$ is the union of the planes of the reflections that do not fix $x$. 

\begin{lemma}\label{invariance}
Let $A$ be a subset of $\R^n$ and $v : A \to \R^n$ a map that satisfy the following conditions.
\begin{enumerate}
\item For all $g\in G$, $g A\cap A\neq\varnothing$ implies $g A=A$.
\item There holds $v(g x)=g v(x)$, for all $x\in A$, $g\in G_A$. 
\end{enumerate}
Then, 
\begin{equation}\label{vtilde}
\tilde v(x) = g v(g^{-1}x),\text{ for all } x\in g A,\ g\in G,
\end{equation}
extends $v$ to an equivariant map $\tilde v:\tilde A\to\R^n$, where $\tilde{A}=\cup_{g\in G}g A$. 
\end{lemma}

\begin{proof}
We first prove that $\tilde v$ is well defined. Assume $x=g_1 x_1=g_2 x_2$, for some $x_1, x_2\in A$ and $g_1, g_2\in G$. Then, we have $x_2=g_2^{-1} g_1 x_1$ and, therefore, $g_2^{-1} g_1 A \cap A \neq \varnothing$, which implies $g_2^{-1} g_1 A = A$ by (i). Thus, $g_2^{-1} g_1\in G_A$ and (ii) yields that $g_2^{-1} g_1v(x_1) = v(x_2)$. From this and the definition \eqref{vtilde} of $\tilde v$, we conclude that
\begin{equation}
\tilde v(x) = g_1 v(g_1^{-1}x) = g_1 v(x_1) = g_2 v(x_2) = g_2 v(g_2^{-1}x) = \tilde v(x).
\end{equation} 
To prove that $\tilde v$ is equivariant, given $x \in \tilde{A}$ and $g\in G$, from \eqref{vtilde} we have that $\tilde v(x) = g_1 v(x_1)$, $\tilde v(gx)=g_2 v(x_2)$, for some $x_1, x_2\in A$ and $g_1, g_2\in G$, such that $x = g_1 x_1$, $g x=g_2x_2$. Therefore, arguing as before, we deduce $v(x_2) = g_2^{-1}g g_1 v(x_1)$ and conclude that
\begin{equation}
\tilde v(gx)=g_2 v(x_2)=g g_1 v(x_1)=g\tilde v(x). \qedhere
\end{equation}
\end{proof}

The following corollary concerns the particular case where $A$ is a ball.
 \begin{corollary}\label{rad-map-cond}
Assume that the ball $B_{x,r}$ satisfies the condition
\begin{equation}\label{bintersection}
B_{x,r}\cap\tilde\Pi_x=\varnothing.
\end{equation}
Let $\alpha : B_{x,r}\to\R$ be a scalar function that depends only on the distance from the center $x$ of $B_{x,r}$ and $w:B_{x,r}\to\R^n$ be a map that satisfies condition (ii) in Lemma \ref{invariance}. Then, \eqref{vtilde} extends the product $v=\alpha w:B_{x,r}\to\R^n$ to an equivariant map $\tilde v:\cup_{g\in G}gB_{x,r}\to\R^n.$
\end{corollary}

\begin{proof}
Since it results $\gamma B_{x,r}=B_{x,r}$, for all $\gamma\in G_x$, the ball $B_{x,r}$ satisfies \eqref{a-cond} in Lemma \ref{algebra} if and only if it has empty intersection with $\pi_\gamma$, for all $\gamma\in\Gamma\setminus\Gamma_x$. From this and Lemma \ref{algebra} it follows that \eqref{bintersection} is a necessary and sufficient condition in order that $A=B_{x,r}$ satisfies condition (i) in Lemma \ref{invariance}. From the assumptions on $\alpha$ and $w$ it is obvious that $v$ satisfies (ii).
\end{proof}

\begin{lemma}\label{d-exists} Let $l_0$ and $\lambda$ be as in Lemma \ref{lemma-sigma}. There exist $d>0$ and $R_0>0$ such that, if $R\geq R_0$, then, for each $x\in D_R$ that satisfies
\begin{equation}\label{bigger-d}
d(x,\partial D_R) \geq d,
\end{equation}
there are $\hat x\in D_R$, for $L\geq L_0 = l_0+\lambda$, such that
\begin{enumerate}
\item $B_{\hat x,L}\subset D_R,$
\item $B_{\hat x,L}\cap\tilde\Pi_{\hat x}=\varnothing,$
\item $x\in B_{\hat x,L-\lambda}.$
\end{enumerate}
\end{lemma}

\begin{proof}
Assume the lemma is false. Then, there are sequences $R_j$, for $x_j \in D_{R_j}$, $1\leq j\leq\ldots$, such that
\begin{equation}\label{lim-rj}
\begin{cases}
\lim_{j\to +\infty}R_j = +\infty,\smallskip\\
\lim_{j \to +\infty} d_j := d(x_j, \partial D_{R_j}) = +\infty,
\end{cases}
\end{equation}
and
\[
B_{\hat x,L}\cap\tilde\Pi_{\hat x}\neq\varnothing, \text{ for all } \hat{x},\, L \text{ such that } L\geq L_0,\,B_{\hat x,L}\subset D_{R_j},\, \vert x_j-\hat x\vert<L-\lambda.
\]

By passing to a subsequence, we can assume that, for each $\gamma\in\Gamma_{a_1}=\Gamma_D$ there exists $\alpha_\gamma\in[0,+\infty]$ such that 
\begin{equation}\label{lim-xi}
\lim_{j\to+\infty}\frac{d(x_j,\pi_\gamma)}{d(x_j,\partial D_{R_j})}=\alpha_\gamma.
\end{equation}

We distinguish two cases.

\medskip \noindent {\em Case 1}. Let $\alpha_\gamma>0$, $\gamma\in\Gamma_{a_1}$. Then, provided $j$ is sufficiently large, \eqref{lim-rj} and \eqref{lim-xi} imply
\begin{equation}\label{case1}
d(x_j, \pi_\gamma) > \frac{1}{2} \bar{\alpha} d_j > L_0, \text{ for } \gamma\in\Gamma_{a_1},
\end{equation}
where $\bar{\alpha} := \min\{\min\{1,\alpha_\gamma\} \mid \alpha_\gamma>0, \text{ for } \gamma\in\Gamma_{a_1}\}$. This shows that the ball $B_{x_j,\frac{1}{2}\bar\alpha d_j}\subset D_{R_j}$ has empty intersection with $\Pi$, in contradiction with the assumptions on the sequences $\{R_j\}$, $\{x_j\}$.

\medskip \noindent {\em Case 2}. Let $\alpha_\gamma=0$, for some $\gamma\in\Gamma_{a_1}$. Let $\pi^0=\cap_{\alpha_\gamma=0} \pi_\gamma$ and let $\xi_j \in \pi^0$ be the orthogonal projection of $x_j$ on $\pi^0$. Then, there is a constant $C>0$ such that 
\begin{equation}\label{d-xj-xi}
\vert x_j - \xi_j \vert \leq C \max_{\alpha_\gamma=0} d(x_j,\pi_\gamma) \leq C d_j\alpha_j^0,
\end{equation}
where 
\[
\alpha_j^0 := \max_{\alpha_\gamma=0} \frac{d(x_j,\pi_\gamma)}{d_j} \to 0, \text{ as } j\to+\infty. 
\]
Therefore, if $\bar\gamma \in \Gamma_{a_1}$ has $\alpha_{\bar\gamma} > 0$, we obtain, for $j$ sufficiently large,
\begin{align}
d(\xi_j,\pi_{\bar\gamma}) \geq d(x_j,\pi_{\bar\gamma})-\vert x_j-\xi_j\vert \geq d_j \left( \frac{1}{2}\alpha_{\bar\gamma}-C\alpha_j^0 \right) > \frac{1}{4}\bar\alpha d_j, \label{case2-1}\\
d(\xi_j,\partial D) \geq d(x_j,\partial D)-\vert x_j-\xi_j\vert  \geq d_j(1- C\alpha_j^0) > \frac{1}{2}d_j.\label{case2-2}
\end{align}

From \eqref{d-xj-xi} and \eqref{case2-1}, \eqref{case2-2}, it follows that, for $j$ sufficiently large, 
$x_j\in B_{\xi_j,\frac{1}{4}\bar\alpha d_j-\lambda}$, the ball $B_{\xi_j,\frac{1}{4}\bar\alpha d_j}$ is contained in $D_{R_j}$  and has empty intersection with $\tilde\Pi_{\xi_j}=\cup_{\gamma\in\Gamma\setminus\Gamma_{\xi_j}} \pi_\gamma$. This is in contradiction with the assumptions on $\{R_j\}$, $\{x_j\}$.
\end{proof}

Assume $R\geq R_0$, with $R_0$ as in Lemma \ref{d-exists} and let
\begin{equation}
 \aleph^R = \{ (x,L) \mid L\geq L_0,\, B_{x,L}\subset D_R,\, B_{x,L}\cap\tilde\Pi_x=\varnothing \}.
\end{equation}
From Lemma \ref{d-exists} and the compactness of the set $\{ x \in D_R \mid d(x,\partial D_R) \geq d \}$ it follows that there is a number $K$ and $(\hat{x}_j, L_j) \in \aleph^R$, for $j=1,\dots,K$, that
depend on $R$ and are such that
\begin{equation}
\{x \in D_R \mid d(x,\partial D_R) \geq d\} \subset \cup_{j=1}^K B_{\hat{x}_j, L_j - \lambda}.
\end{equation}
Define the set  $\Omega^R \subset D_R$ by 
\begin{equation}\label{omegar-def}
\Omega^R = \cup_{j=1}^K B_{\hat{x}_j, L_j - \lambda}.
\end{equation}
The set $\Omega^R$ is open and we can assume that the sequence $\{B_{\hat x_j,L_j - \lambda}\}_{j=1}^{K}$ contains $g B_{\hat x_j,L_j}$, for all $g\in G_D$, $j=1,\dots,K$, so that 
 \begin{equation}
G_{\Omega^R}=G_D=G_{a_1}. 
 \end{equation}

\subsection{The replacement lemmas}\label{replacement-lemmas}
Let $\bar{q}^\prime > 0$ be the constant in Lemma \ref{lemma-sigma} and let $c>0$ as before in \eqref{final-comp}. Assume $R\geq R_0$ and $\Omega^R$ as in \eqref{omegar-def}.

\begin{lemma}\label{q-definition} 
Let $\mathsf{q} : \Omega^R \to \R$ be the solution of 
\begin{equation}\label{final-comp2}
\begin{cases}
\Delta \mathsf{q} = c^2 \mathsf{q}, &\text{in } \Omega^R,\smallskip\\
\mathsf{q} = \bar q^\prime, &\text{on } \partial\Omega^R,
\end{cases}
\end{equation}
Then,
\begin{equation}\label{final-comp3}
\mathsf{q}(g x) = \mathsf{q}(x),\text{ for all } g\in G_{\Omega^R}= G_D= G_{a_1}.
\end{equation}
Moreover, 
\begin{equation}\label{ex-bound1}
\mathsf{q}(x) \leq K \mathrm{e}^{-k d(x,\partial\Omega^R)}, \text{ for } x\in\Omega^R,
\end{equation}
and, in particular, if $d>0$ is as in Lemma \ref{d-exists},  
\begin{equation}\label{ex-bound2}
\mathsf{q}(x) \leq K \mathrm{e}^{-k d(x,\partial D_R)}, \text{ in } B_{x,d}\subset D_R.
\end{equation}
for some constants $K, k>0$ independent of $R$.
\end{lemma}

\begin{proof} The invariance follows from uniqueness. The maximum principle implies $\mathsf{q} \leq \bar{q}^\prime$. It follows that if $\varphi$ is the solution of equation \eqref{final-comp2} on the ball with center $x$ and radius $d(x,\partial\Omega^R)$ with boundary condition $\varphi = \bar{q}$, we have $\mathsf{q} \leq \varphi$. This and the estimate \eqref{ex-bound} in Lemma \ref{lemma-phi1} imply \eqref{ex-bound1} for some $K, k>0$ independent of $R$. The last estimate follows from $d(x,\partial D_R)\leq d(x,\partial\Omega^R)+d$, after changing $K$ to $K \mathrm{e}^{kd}$.
\end{proof} 

\begin{lemma} \label{comparison-lemma-1}
Let $A\subset D_R$ be an open connected set with Lipschitz boundary and let $\Phi$ the solution of the problem  
\begin{equation}\label{comparison-problem1}
\begin{cases}
\Delta \Phi = 0, &\text{in } A,\smallskip\\
\Phi = f, &\text{on } \partial A,
\end{cases}
\end{equation}
for a smooth function $f: \partial A \to \R$. Assume that $f>0$ so that
\begin{equation*}
 \Phi_m =\min_{x\in A}\Phi(x)>0.
\end{equation*}
Assume also that $A$, $f$, $u\in\mathcal{U}^{\rm Pos}$, and $0<b\leq\Phi_m$ satisfy the following.
\renewcommand{\labelenumi}{(\alph{enumi})}
\begin{enumerate}
\item $A$ satisfies (i) in Lemma \ref{invariance}.
\item $f$  is the trace of a smooth map $f^*$ that satisfies
\[ f^*(g x) = f^*(x),\text{ for all } x\in A,\, g\in G_A.\]
\item $q^u \in L^\infty (D_R)$ and  $q^u|_{\partial A} \leq f$, on $\partial A$.
\item The set $A_b := \{ x \in A \mid q^u(x) > b \}$ is open and $\nu^u|_{\overline{A_b}}$ is $C^1$ smooth. 
\end{enumerate}
Then, there is a $v \in\mathcal{U}^{\rm Pos}$ such that
\renewcommand{\labelenumi}{(\roman{enumi})}
\begin{enumerate}
\item $\nu^v = \nu^u$, on $D_R\setminus S_u,\, S_u= \{ x \in D_R \mid q^u=0 \}$.
\item $q^v \leq \Phi$, in $A$.
\item $v|_{B_R \setminus\tilde A} = u|_{B_R \setminus\tilde A}$,\, $\tilde A=\cup_{g\in G}g A$.
\item $J_{B_R}(v) \leq J_{B_R} (u)$.
\end{enumerate}
\end{lemma}

\begin{proof}
Lemma \ref{thre-two} implies the existence of a minimizer $\rho\in W^{1,2}(A_b)\cap L^\infty(A_b)$ of $\mathcal{K}_{A_b}$ on the subset of the functions that satisfy the Dirichlet condition
\begin{equation}\label{b-cond}
\rho = q^u, \text{ on } \partial A_b,
\end{equation}
and the invariance condition 
\begin{equation}\label{rho-invariance}
 \rho(g x)= \rho(x),\text{ for } x\in A_b,\, g\in G_{A_b}.
\end{equation}

Let $A_b^* := \{x \in A_b \mid \rho(x) > \Phi\}$. Then we have that $\rho$ satisfies
\begin{equation}\label{rho}
\begin{split}
\int_{A_b^*} \Big\{ \langle \tilde{u}_{qq}(\rho, \nu^u), \tilde{u}_q(\rho, &\,\nu^u) \rangle |\nabla \rho|^2 + \sum_{j=1}^{n} \langle \tilde{u}_{q\nu}(\rho, \nu^u) \nu_{x_j}^u, \tilde{u}_\nu(\rho, \nu^u) \nu_{x_j}^u \rangle \Big\} \eta\, \dd x \\&+ \int_{A_b^*} \langle \tilde{u}_{q}(\rho, \nu^u), \tilde{u}_q(\rho, \nu^u) \rangle \nabla \rho \nabla \eta\, \dd x = 0, 
\end{split}
\end{equation}
for all $\eta \in W^{1,2}_{0}(A_b)\cap L^\infty(A_b)$ that satisfy \eqref{rho-invariance} and vanish on $\{\rho\leq\Phi\}$. Taking $\omega = \omega_j$ in \eqref{quadratic}, with $\alpha = \rho_{x_j}$, $\beta = 1$, and $t = \nu_{x_j}^u$, we obtain, for $\eta \geq 0$,
\begin{equation}\label{omegaj}
\bigg( \sum_{j=1}^{n} \omega_j \bigg) \eta = \bigg(\! - \langle \tilde{u}_{qq}, \tilde{u}_q \rangle |\nabla \rho|^2 + \sum_{j=1}^{n} \langle \tilde{u}_{q\nu} \nu_{x_j}^u, \tilde{u}_\nu \nu_{x_j}^u \rangle - 2 \sum_{j=1}^{n} \langle \tilde{u}_{q\nu} \nu_{x_j}^u, \tilde{u}_q \rangle \rho_{x_j} \bigg) \eta \geq 0.
\end{equation}
Integrating \eqref{omegaj} and subtracting from \eqref{rho} gives
\begin{equation}\label{weak-rho}
\int_{A_b^*} \nabla \rho \nabla ( \langle \tilde{u}_q, \tilde{u}_q \rangle\eta ) \dd x \leq 0,
\end{equation}
for all nonnegative $\eta \in W_{0}^{1,2}(A_b)\cap L^\infty(A_b)$ that satisfy \eqref{rho-invariance} and vanish on the set $\{\rho\leq\Phi\}$. On the other hand, the definition of $\Phi$ implies
\begin{equation}\label{weak-phi}
\int_{A} \nabla\Phi\nabla\zeta \dd x = 0, 
\end{equation}
for all $\zeta \in W_{0}^{1,2}(A)$. We take $\eta = {(\rho-\Phi)^+}/{\langle \tilde{u}_q, \tilde{u}_q \rangle}$ and $\zeta=(\rho-\Phi)^+$ and subtract \eqref{weak-phi} from \eqref{weak-rho} to obtain
\begin{equation*}
\int_{ A_b^*} \vert\nabla(\rho-\Phi)^+\vert^2 \dd x \leq 0,
\end{equation*}
and, therefore, using also $\rho\leq\Phi$ for $x\in A_b\setminus A_{b}^{*}$,
\begin{equation}\label{rho-phi}
\rho \leq \Phi, \text{ in } A_b.
\end{equation}

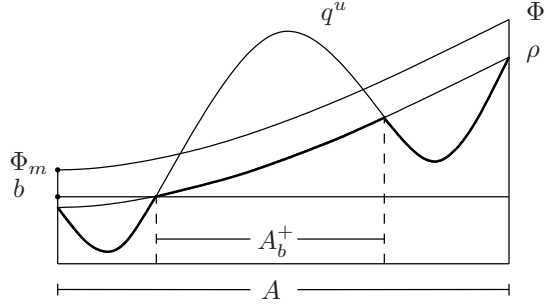
\begin{figure}
\begin{center}
\begin{picture}(0,0)%
\includegraphics{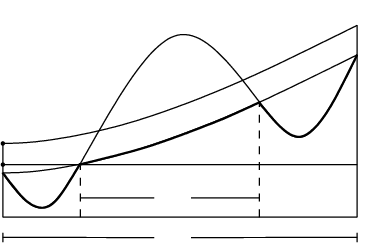}%
\end{picture}%
\setlength{\unitlength}{4144sp}%
%
\begin{picture}(2834,1913)(1104,12)
\put(2342, 66){$A$}%
\put(2321,369){$A^{+}_{b}$}%
\put(3928,1739){$\Phi$}%
\put(3928,1514){$\rho$}%
\put(838,839){$\Phi_m$}%
\put(860,675){$b$}%
\put(2701,1739){$q^u$}%
\end{picture}%
\caption{The functions $\rho$ and $\Phi$.}
\label{figure:figure2}
\end{center}
\end{figure}

Define $q^v : A \to \R$  by setting
\begin{equation}\label{sv-def}
q^v(x) = 
\begin{cases}
\min\{ \rho(x), q^u(x) \}, &\text{for } x \in A_b,\smallskip\\
q^u(x), &\text{for } x \in A\setminus A_b,
\end{cases}
\end{equation}
and observe that (ii) follows from this, from the inequality and from \eqref{rho-phi} and $q^u\leq b\leq\Phi_m$ in $A\setminus A_b$. Observe also that 
\begin{equation}\label{qv-qu}
q^v = q^u,\text{ for } x \in\partial A.
\end{equation}
 $A\subset D_R$ implies $G_A\subset G_D$. This and \eqref{nueq-symmetry} imply that $q^u$ and therefore also $q^v$ satisfies \eqref{rho-invariance}.  It follows that, if we set
\begin{equation}\label{nuv-nuu}
\nu^v  = \nu^u , \text{ on }  A\setminus S_u,
\end{equation}
and recall \eqref{nu-symmetry} and \eqref{nueq-symmetry}, then the map $v : A\to\R^n$ defined by 
\[ v(x) =  
\begin{cases}
\tilde{u}( q^v(x); \nu^u(x)), &\text{for } x\in A\setminus S_u,\smallskip\\
0, &\text{for } x\in A\cap S_u,
\end{cases} \]
satisfies (i) and (ii) in Lemma \ref{invariance}. Therefore, $v$ can be extended to an equivariant map $v:\tilde A \to \R^n$, $\tilde A = \cup_{g\in G}g A$. From \eqref{qv-qu} and \eqref{nuv-nuu} we see that $v$ and $u$ have the same trace on $\partial\tilde A$. It follows that, if we extend $v$ to the whole $B_R$ by setting $v=u$, on $B_R\setminus\tilde A$, then we have a well-defined equivariant map $v \in W_{\mathrm{E}}^{1,2}(B_R;\R^n)$. This in particular proves (iii). Moreover, $v$ is a positive map because $u$ is and, by definition, $q^v \leq q^u$. It remains to prove (iv). We argue as follows. The definition of $v$ implies
\[ J_{B_R} (v) = J_{\tilde A} (v) + J_{B_R \setminus\tilde A} (u), \text{ with } J_{\tilde A} (v) = \frac{\vert G\vert}{\vert G_A\vert} J_A (v).\]
Let ${A_b}^+\subset A_b$ be the subset ${A_b}^+ := \{ x \in A_b \mid q^u(x) > \rho(x) \}$ and observe that 
\begin{equation*}
J_{{A_b}^+} (v) = \mathcal{K}_{{A_b}^+} (\rho) + \mathcal{V}_{{A_b}^+} (\rho)\leq
\mathcal{K}_{{A_b}^+} (q^u) + \mathcal{V}_{{A_b}^+} (q^u)=J_{{A_b}^+} (u),
\end{equation*}
where we have used the minimality of $\rho$ and \eqref{v-pos}. Therefore, recalling that $v=u$ on $A\setminus {A_b}^+$ we obtain  
\begin{equation*}
J_A (v) = J_{{A_b}^+} (v) + J_{A \setminus {{A_b}^+}} (u)\leq J_A (u) \qedhere
\end{equation*}
\end{proof}
 
\begin{lemma}\label{comparison-lemma-2}
Let $c, \bar q$ be as in Hypothesis \ref{h1} and $A$ as in Lemma \ref{comparison-lemma-1}, and let $\Psi$ be the solution of the problem
\begin{equation}\label{comparison-problem2}
\begin{cases}
\Delta \Psi = c^2 \Psi, &\text{in } A,\smallskip\\
\Psi = h, &\text{on } \partial A,
\end{cases}
\end{equation}
for a smooth function $h: \partial A \to \R$. Assume that $h>0$ so that
\begin{equation*}
 \Psi_m =\min_{x\in A}\Psi(x)>0.
\end{equation*}
Assume that $A$, $h$, $u\in\mathcal{U}^{\rm Pos}$, and $0<b\leq\Psi_m$  satisfy the following.
\renewcommand{\labelenumi}{(\alph{enumi})}
\begin{enumerate}
\item $A$ satisfies (i) in Lemma \ref{invariance}.
\item $h$ is the trace of a smooth map $h^*$ that satisfies
\[ h^*(g x) = h^*(x), \text{ for all } x\in A,\, g\in G_A.\]
\item There holds
\[ q^u(x) \leq \bar{q}, \text{ for } x \in A, \]
and
\[ q^u|_{\partial A} \leq h\leq\bar{q}, \text{ on } \partial A.\]
\item The set $A_b := \{ x \in A \mid q^u(x) > b \}$ is open and $\nu^u|_{\overline{A_b}}$ is $C^1$ smooth.
\end{enumerate}
Then, there is a $v \in\mathcal{U}^{\rm Pos}$ such that
\renewcommand{\labelenumi}{(\roman{enumi})}
\begin{enumerate}
\item $\nu^v=\nu^u$, on $D_R\setminus S_u$.
\item $q^v \leq \Psi$,\ in $A$.
\item $v|_{B_R \setminus\tilde A} = u|_{B_R \setminus\tilde A}$,\, $\tilde A=\cup_{g\in G}g A$.
\item $J_{B_R}(v) \leq J_{B_R} (u)$.
\end{enumerate}
\end{lemma}

\begin{proof}
The proof parallels the proof of Lemma \ref{comparison-lemma-1}. We minimize the functional $\mathcal{E}_{A_b}$ on the weakly closed subset of $ W^{1,2}(A_b)$ defined by \eqref{b-cond} and \eqref{rho-invariance} in the proof of Lemma \ref{comparison-lemma-1} and obtain that, if $\rho$ is a minimizer of $\mathcal{E}_{A_b}$ and $A_b^*=\{x\in A_b \mid \rho>\Psi\}$, then we have  
\begin{equation}\label{weak-rho-eta}
\int_{A_b^*} \left\{ \nabla \rho \nabla ( \langle \tilde{u}_q(\rho; \nu^u), \tilde{u}_q(\rho; \nu^u) \rangle\eta ) + V_q(\rho,\nu^u)\eta \right\} \dd x \leq 0,
\end{equation}
for all nonnegative $\eta \in W_{0}^{1,2}(A_b)\cap L^\infty(A_b)$ that satisfy \eqref{rho-invariance} and vanish on the set $\{\rho\leq\Psi\}$. From \eqref{weak-rho-eta} and \eqref{vc-pos} it follows
\begin{equation}\label{weak-rho-eta1}
\int_{A_b^*} \left\{ \nabla \rho \nabla ( \langle \tilde{u}_q(\rho; \nu^u), \tilde{u}_q(\rho; \nu^u) \rangle\eta ) + c^2\langle \tilde{u}_q(\rho; \nu^u), \tilde{u}_q(\rho; \nu^u) \rangle\rho\eta \right\} \dd x \leq 0,
\end{equation} 
From \eqref{comparison-problem2} we also have
\begin{equation}\label{weak-psi}
\int_{A} \nabla\Psi\nabla\zeta+c^2\Psi\zeta=0,\text{ for } \zeta\in W_{0}^{1,2}(A).
\end{equation}
If we set $\eta={(\rho-\Psi)^+}/{\langle \tilde{u}_q(\rho, \nu^u), \tilde{u}_q(\rho, \nu^u) \rangle}$ in \eqref{weak-rho-eta1} and subtract \eqref{weak-psi} with $\zeta=(\rho-\Psi)^+$ from \eqref{weak-rho-eta1}, we obtain
\begin{equation}
\int_{A_b^*} \vert\nabla(\rho-\Psi)^+\vert^2 + c^2{(\rho-\Psi)^+}^2 \dd x \leq 0.
\end{equation}
From this it follows that $A_{b}^{*}$ has zero measure and therefore we have
\begin{equation}
\rho \leq \Psi, \text{ in } A_b.
\end{equation}
The remaining proof is analogous to the proof of Lemma \ref{comparison-lemma-1}.
\end{proof}

\begin{proposition}\label{sigmacomparison}
Let $\lambda$, $l_0$, $l \geq l_0$, $\delta$, $L=l+\lambda$, and $\sigma$ be as in Lemma \ref{lemma-sigma}. Let
\begin{equation*}
\sigma_m =\min_{x\in B_L}\sigma(x)>0.
\end{equation*}
and set $\sigma_{\hat x} := \sigma(\cdot-\hat x)$. Assume that $B_{\hat x,L}\subset D_R$ satisfies 
$B_{\hat x,L}\cap\tilde\Gamma_{\hat x}=\varnothing$ and also assume that $u\in\mathcal{U}^{\rm Pos}$ and $0<b\leq\sigma_m$ satisfy
\renewcommand{\labelenumi}{(\alph{enumi})}
\begin{enumerate}
\item $q^u\leq \overline{Q}$, for $x\in \overline{B_L}$ (cf.\ \eqref{quasifinal-comp5}),
\item $q^u \leq \bar q$, for $x\in \overline{B_{\hat x,L-\lambda}}$,
\item the set $A_b^\circ := \{ x \in D_R \mid q^u(x)>b \}$ is open and $\nu^u|_{\overline{A_b^\circ}}$ is $C^1$ smooth.
\end{enumerate}
Then, there exists $v\in\mathcal{U}^\mathrm{Pos}$ such that
\renewcommand{\labelenumi}{(\roman{enumi})}
\begin{enumerate}
\item $\nu^v=\nu^u$, on $D_R\setminus S_u$,
\item $q^v \leq \sigma_{\hat x}$, for $x\in \overline{B_{\hat x,L}}$,
\item $v=u$, for $x\in B_R\setminus\tilde B_{\hat x,L}$, $\tilde B_{\hat x,L} = \cup_{g\in G} B_{\hat x,L}$,
\item $J_{B_R}(v) \leq J_{B_R}(u).$
\end{enumerate}
\end{proposition}

\begin{proof}
Set $\varphi_{j,\hat x} = \varphi(\cdot-\hat x)$, for $j=1,2$, and $\vartheta_{\hat x} = \vartheta(\cdot-\hat x)$ with $\varphi_j$, for $j=1,2$, as in \eqref{final-comp}, \eqref{quasifinal-comp}, and $\vartheta$ as in \eqref{small-comp}. From Lemma \ref{comparison-lemma-1}, with $A=B_{\hat x,L}\setminus\overline{B_{\hat x,L-\lambda}}$, $A_b=A_b^\circ\cap A$, and $\Phi=\varphi_{2,\hat x}$ and also utilizing Corollary \ref{rad-map-cond}, we can replace $u$ with a map $v\in W_{\mathrm{E}}^{1,2}(B_R;\R^n)$ that satisfies (i), (iii), and (iv), and $q^v\leq\varphi_{2,\hat x},$ for $x\in\overline{B_{\hat x,L}\setminus B_{\hat x,L-\lambda}}$. Similarly, from Corollary \ref{rad-map-cond} and Lemma \ref{comparison-lemma-2}, with $A=B_{\hat x,L-\lambda}$, $A_b=A_b^\circ\cap A$, and $\Psi=\varphi_{1,\hat x}$, we can replace $u$ with a map $v\in W_{\mathrm{E}}^{1,2}(B_R;\R^n)$ that satisfies (i), (iii), and (iv), and $q^v\leq\varphi_{1,\hat x}$ in $\overline{B_{\hat x,L-\lambda}}$. Finally, a further application of Corollary \ref{rad-map-cond} and Lemma \ref{comparison-lemma-1}, with $A=B_{\hat x,L-\lambda+\delta}\setminus\overline{B_{\hat x,L-\lambda-\delta}}$, $A_b=A_b^\circ\cap A$, and $\Phi=\vartheta_{\hat x}$, concludes the proof.
\end{proof}

\begin{proposition}\label{q-comparison}
Assume $R\geq R_0$, $\Omega^R\subset D_R$, and $\mathsf{q} : \Omega^R \to \R$ as in Lemma \ref{q-definition}. Let 
\begin{equation*}
\mathsf{q}_m = \min_{x\in\Omega^R} \mathsf{q}(x) > 0.
\end{equation*}
Assume that $u \in W_{\mathrm{E}}^{1,2}(B_R;\R^n)$ and $0<b\leq\mathsf{q}_m$ satisfy
\renewcommand{\labelenumi}{(\alph{enumi})}
\begin{enumerate}
\item $q^u \leq \bar{q}^{\prime}$, for $x\in\overline{\Omega^R}$, where $\bar{q}^{\prime} < \bar{q}$ is the constant in Lemma \ref{lemma-sigma},
\item the set $A_b := \{x \in A \mid q^u(x) > b \}$ is open and $\nu^u|_{\overline{A_b}}$ is $C^1$ smooth.
\end{enumerate}
Then, there is a $v\in W_{\mathrm{E}}^{1,2}(B_R;\R^n)$ such that
\renewcommand{\labelenumi}{(\roman{enumi})}
\begin{enumerate}
\item $\nu^v = \nu^u$, on $D_R\setminus S_u$,
\item $q^v \leq \mathsf{q}$, for $x \in \overline{\Omega^R}$,
\item $v=u$, for $x\in B_R\setminus\tilde\Omega^R$, $\tilde\Omega^R = \cup_{g\in G}\Omega^R$,
\item $J_{B_R}(v)\leq J_{B_R}(u).$
\end{enumerate}  
\end{proposition}

\begin{proof}
It suffices to apply Lemma \ref{comparison-lemma-2} with $A = \Omega^R$ and $\Psi = \mathsf{q}$ 
and Lemma \ref{invariance}, taking into account that $G_{\Omega^R}=G_D=G_{a_1}$.
\end{proof}

\section{Proof of Theorem \ref{theorem1}}\label{proof}
Let $R> R_0$, $\Omega^R$, $\bar q$, $\bar{q}^{\prime}<\bar{q}$, and $F_R$ be as before. Fix a number $q_0 \in (\bar{q}^{\prime}, \bar{q})$ and define the {\em admissible} set $\mathcal{A}^R \subset  W_{\mathrm{E}}^{1,2}(B_R;\R^n)$ by setting
\begin{equation}\label{admissible-set}
\mathcal{A}^R := \left\{ u \in W_{\mathrm{E}}^{1,2}(B_R,\R^n) \mid u(\overline{F_R}) \subset \overline{F};\; q^u\leq q_0,\text{ for } x\in\overline{\Omega^R}+B_{{\delta^{\prime}}\!/{2}} \right\},
\end{equation}
where $\delta^{\prime}$ is the constant in Lemma \ref{lemma-sigma}. 

\medskip
 
\noindent {\em Step 1. There exists a minimizer $u_R \in W^{1,2}_{\mathrm{E}}(B_R; \R^n)$ of the problem
\begin{equation}\label{minimization}
\min_{u\in\mathcal{A}^R}J_{B_R}(u). 
\end{equation}
Moreover,
\begin{equation}\label{pointwise-bound}
|u| \leq M,
\end{equation}
where $M$ is the constant in Hypothesis \ref{h2}.}

\medskip

For $u\in W_{\mathrm{E}}^{1,2}(B_R;\R^n)$ we have $J_{B_R}(u)=J_{\{\vert u\vert>M\}}(u) + J_{B_R\setminus \{\vert u\vert>M\}}(u)$. Set $\nu= u /\vert u\vert$, for $|u| \neq 0$; then
\begin{align*}
J_{\{\vert u\vert>M\}}(u) &= \int_{\{\vert u\vert>M\}} \Bigg\{ \frac{1}{2} \Bigg( \vert\nabla\vert u\vert\vert^2+\vert u\vert^2\sum_{j=1}^n\langle\nu_{x_j},\nu_{x_j}\rangle \Bigg)+W(\vert u\vert\nu) \Bigg\} \dd x\\
&> \int_{\{\vert u\vert>M\}} \Bigg\{ \frac{1}{2} M^2 \sum_{j=1}^n \langle\nu_{x_j},\nu_{x_j}\rangle + W(M\nu) \Bigg\} \dd x \\
&= J_{\{\vert u\vert>M\}}(M\nu),
\end{align*}
where we have also used Hypothesis \ref{h2}. This proves that minimizers satisfy the $L^\infty$ bound (\ref{pointwise-bound}) and therefore that we can restrict to the subset of $\mathcal{A}^R$ of the maps $u$ that satisfy 
\begin{equation}\label{bobo}
q^u \leq \overline{Q}, \text{ for } x \in D_R, \text{ where } \overline{Q}=\max_{u\in\overline{D},\, \vert u\vert\leq M}Q(u).
\end{equation}
Define
\begin{equation}\label{u-aff-def}
u_{\mathrm{aff}}(x) := 
\begin{cases}
d(x;\partial D) a_1, &\text{for } x \in D_{R} \text{ and } d(x;\partial D) \leq 1,\smallskip\\
a_1,  &\text{for } x \in D_{R} \text{ and } d(x;\partial D) \geq 1.
\end{cases}
\end{equation}
The map $u_{\mathrm{aff}}$ trivially satisfies condition (ii) in Lemma \ref{invariance} and therefore extends to an equivariant map on $B_R$. Clearly, $u_\mathrm{aff} \in \mathcal{A}^R$. By the nonnegativity of $W$ and a simple calculation,
\begin{equation}
0 \leq \inf_{u\in\mathcal{A}^R} J_{B_R}(u) < J_{B_R} (u_\mathrm{aff}) < C R^{n-1},
\end{equation}
for some constant $C$ independent of $R$.

Let $\{ u_k \}_{k=1}^\infty \subset \mathcal{A}^R$ be a minimizing sequence.  Without loss of generality,  we may assume that \eqref{pointwise-bound} holds for each value of $k$. We have
\begin{equation}
\frac{1}{2} \int_{B_R} |\nabla u_k|^2 \dd x < J_{B_R}(u_\mathrm{aff}) < C R^{n-1} \quad\text{and}\quad \int_{B_R} |u_k|^2 \dd x < C_R,
\end{equation}
where $C_R$ denotes a constant depending on $R$. By standard arguments, we obtain, possibly along a subsequence,
\begin{equation}
u_k \to u_R, \text{ a.e.},
\end{equation}
where $u_R\in\mathcal{A}^R$ is a minimizer of \eqref{minimization}. Clearly, $q^{u_R} \leq q_0$ on $\overline{\Omega^R}+B_{{\delta^\prime}\!/{2}}$ and $|u_R(x)|\leq M$ on $B_R$. This finishes the proof of Step 1.

\medskip \noindent {\em Step 2.  The minimizer $u_R$, for $R\geq R_0$, satisfies 
\begin{eqnarray}\label{ur-equilibrium}
u(\cdot,t,u_R)=u_R,\quad t>0,
\end{eqnarray} 
where, as before, $u(\cdot,t,u_R)$ is the solution of (\ref{evolution-problem}) with initial condition $u_0=u_R$.}

\medskip

Before proving \eqref{ur-equilibrium} we observe that \eqref{ur-equilibrium} implies that $u_R$ is a classical solution of $\Delta u - W_u(u) = 0$ on the ball $B_R$ with Neumann boundary condition. Moreover, by Theorem \ref{theorem-2-1}, $u_R\in\mathcal{U}_0^{\rm Pos}$. 

We argue by contradiction. Assume that \eqref{ur-equilibrium} does not hold. There is a sequence $\tilde{t}>0$ that converges to $0$ and it is such that 
\begin{equation}\label{ur-energy}
J_{B_R}(\tilde{u}_R) < J_{B_R}(u_R),
\end{equation}
where we have set $\tilde{u}_R = u(\cdot,\tilde{t},u_R)$. If $\tilde{t}>0$ is sufficiently small, we also have 
\begin{equation}\label{less-q-bar}
q^{\tilde{u}_R} \leq\bar{q}, \text{ for } x\in\overline{\Omega^R}.
\end{equation}
This follows from $u_R \in \mathcal{A}^R$, which implies $q^{u_R} \leq q_0< \bar{q}$, for $x\in\overline{\Omega^R}+B_{{\delta^\prime}\!/{2}}$. We now fix $\tilde{t}$ as above. From the definition of $\mathcal{A}^R$, Theorem \ref{theorem-2-1}, and the fact that \eqref{evolution-problem} preserves the pointwise bound \eqref{pointwise-bound}, it follows that
\begin{equation}
\tilde{u}_R\in\mathcal{U}_0^{\rm Pos} \text{ and } q^{\tilde{u}_R}\leq\overline{Q},\text{ for } x\in\overline{D_R}.
\end{equation}
Let $\sigma_m$ be as in Proposition \ref{sigmacomparison} and let $\bar{L} = \max\{L \mid B_{x,L}\subset D_R\}$. Observe that $\sigma_m$ is a nonincreasing function of $L\in[L_0,\bar{L}]$ and that there is a $\bar{\sigma}>0$ such that 
\begin{equation}
\sigma_m \geq\bar{\sigma},\quad L\in[L_0,\bar{L}].
\end{equation}
Since $\tilde{u}_R\in C^2(\overline{B_R};\R^n)$, given $0<b\leq\bar{\sigma}$, the set $A_b^\circ = \{ x \in D_R \mid q^{\tilde{u}_R} > b \}$ is open and $\nu^{\tilde{u}_R}\vert_{\overline{A_b^\circ}}$ is $C^2$.
Assume that $q_0 < q^{\tilde{u}_R} \leq \bar{q}$ on some subset of $\overline{\Omega^R}$ and let $B_{\hat x_j,L_j}$, for $j=1,\dots,K$ be the sequence in the definition \eqref{omegar-def} of $\Omega^R$.   Since we also have that $B_{\hat x,L}\cap\tilde\Pi_{\hat x}=\varnothing$, we see that $\tilde{u}_R$, $B_{\hat x_1,L_1}$, $A_b^\circ$, satisfy all assumptions of Proposition \ref{sigmacomparison}, therefore, recalling that $q^v\leq\sigma_{\hat{x}}$ implies $q^v\leq\bar{q}'<q_0$, for $x\in B_{\hat x_1,L_1+\delta^\prime-\lambda}$, by applying Proposition \ref{sigmacomparison} we conclude that there exists a $v_1\in\mathcal{U}_0^{\rm Pos}$ with $J_{B_R}(v_1)\leq J_{B_R}(\tilde{u}_R)<J_{B_R}(u_R)$ and 
\begin{equation}
q^{v_1}\leq \bar{q}^\prime<q_0, \; x\in B_{\hat x_1,L_1+\delta^\prime-\lambda}.
\end{equation} 
The map $v_1$ given by Proposition \ref{sigmacomparison} satisfies the same assumptions as $\tilde u_R$, therefore we can apply again Proposition \ref{sigmacomparison} with $v_1$, $B_{\hat x_2,L_2}$, $A_b^\circ$ to obtain the existence of a map $v_2$ that belongs to $\mathcal{U}_0^{\rm Pos}$, has $q^{v_2}\leq q^{v_1}$, and satisfies
\begin{equation}
q^{v_2} \leq \bar{q}^\prime < q_0, \text{ for } x\in\cup_{j=1}^2 B_{\hat x_j,L_j+\delta^\prime-\lambda}
\end{equation} 
together with $J_{B_R}(v_2)\leq J_{B_R}(v_1)\leq J_{B_R}(\tilde{u}_R)<J_{B_R}(u_R)$. After $K$ similar steps we end up with a map $v_K\in\mathcal{U}_0^{\rm Pos}$ that satisfies
\begin{equation}\label{bound-ur} 
q^{v_K} \leq \bar{q}^\prime<q_0, \text{ for } x\in\cup_{j=1}^K B_{\hat x_j,L_j+\delta^\prime-\lambda}
\end{equation}
together with all the other requirements for membership in $\mathcal{A}^R$ and moreover,
\begin{equation}
J_{B_R}(v_K)\leq J_{B_R}(\tilde{u}_R)<J_{B_R}(u_R).
\end{equation}   
This contradicts the minimality of $u_R$ and establishes \eqref{ur-equilibrium}. The proof of Step 2 is concluded.
 
\medskip \noindent {\em Step 3. (Conclusion).} From \eqref{bound-ur} it follows that we can apply Proposition \ref{q-comparison} to conclude that $q^{u_R}(x)\leq\mathsf{q}(x)$, for $x\in\overline{\Omega^R}$ and therefore that, by Lemma \ref{q-definition} and \eqref{bobo},
\begin{equation}\label{ex-uniform}
  \vert u_R(x)-a_1\vert \leq K e^{-k d(x,\partial D_R)}, \text{ for } x\in D_R,
\end{equation}
for some constants $k,\, K>0$ independent of $R$.
As remarked earlier, $u_R$ satisfies
\begin{equation}\label{e-l}
\Delta u - W_u(u) = 0, \text{ on }  B_R, \text{ for } R>R_0,
\end{equation}
and the exponential bound \eqref{ex-uniform}.

Finally, the uniform bound \eqref{pointwise-bound} and elliptic regularity, via a diagonal argument, allow us to pass to the limit along a subsequence in $R$ and capture a function
\begin{equation}
u(x) = \lim_{R_n \to\infty} u_{R_n}.
\end{equation}
The uniform bounds \eqref{bound-ur}, \eqref{ex-uniform} imply that the limit function $u$ satisfies the exponential bound in Theorem \ref{theorem1} and that it is a solution of 
\begin{equation}\label{soln}
\Delta u - W_u(u) = 0, \text{ on } \R^n.
\end{equation}

Finally, we argue the strong positivity for $u$. We already know that $u \in \mathcal{U}^{\mathrm{Pos}}$. Take now an open connected set $U \subset F$ which contains some points far enough from $\partial D$, so that by the exponential bound,
\begin{equation}\label{posss}
u (U) \cap F \neq \varnothing.
\end{equation}
As in the proof of Theorem \ref{theorem-2-1}, let
\begin{equation}
z(x) = \langle u(x), \eta_\gamma \rangle, \text{ in } U.
\end{equation}
Then,
\[
\begin{cases}
\Delta z + cz = 0, &\text{in } U, \quad \text{(by \eqref{soln})}\smallskip\\
z \geq 0, &\text{in } U. \quad \text{(by positivity)}
\end{cases}
\]
By a well-known variant of the strong maximum principle, there holds $z > 0$ in $U$, unless $z \equiv 0$. Triviality however is excluded by \eqref{posss}. From this, strong positivity follows.

This concludes the proof of Theorem \ref{theorem1}.\qed

\section*{Acknowledgments}
The authors would like to express their gratitude to the group attending the Applied Analysis and PDE's seminar of the Department of Mathematics at the University of Athens, in particular to Vassilis Papanicolaou and Achilles Tertikas, for useful discussions on the contents of the paper. Thanks are also due to Peter Bates for his numerous suggestions that improved the presentation and to Christos Athanasiadis for his help with the literature on reflection groups. Special thanks also go to our graduate students Apostolos Damialis and Nikolaos Katzourakis for their help with the manuscript. Finally, we would like to acknowledge a stimulating discussion with Ha\"\i m Brezis during the development of the results in the paper.

\nocite{*}
\bibliographystyle{plain}

\end{document}